\documentclass[11pt]{article}
\usepackage [dvips]{graphics}
\usepackage[centertags]{amsmath}
\usepackage{amsfonts}
\usepackage{amssymb}
\usepackage{amsthm}
\usepackage{newlfont}
\usepackage{stmaryrd}
\usepackage{color,epsfig}
\usepackage{amsfonts,graphicx,psfrag}
\usepackage{graphicx}
\usepackage{float}

\usepackage{txfonts}
\usepackage[noblocks]{authblk}
\usepackage{mathrsfs}

\theoremstyle{plain}
\newtheorem{thm}{Theorem}[section]
\newtheorem{cor}[thm]{Corollary}
\newtheorem{lem}[thm]{Lemma}
\newtheorem{prop}[thm]{Proposition}

\newtheorem{rem}[thm]{Remark}
\newtheorem{defi}[thm]{Definition}

\def\sqr#1#2{{\vcenter{\vbox{\hrule height.#2pt
              \hbox{\vrule width.#2pt height#1pt \kern#1pt \vrule
width.#2pt}
              \hrule height.#2pt}}}}
\def\Sp{{\mathrm {Sp}}}

\def\lb{\label}

\def\d{{d\over dt}}

\def\3n{\negthinspace \negthinspace \negthinspace }
\def\2n{\negthinspace \negthinspace }
\def\1n{\negthinspace }

\def\R{{\mathbb R}}
\def\D{\Delta}

\def\ga{{\gamma}}

\def\cA{{\cal A}}
\def\cB{{\cal B}}

\def\cF{{\cal F}}

\def\cH{{\cal H}}
\def\cI{{\cal I}}

\def\cN{{\cal N}}

\def\no{\noindent}

\def\bs{\bigskip}

\def\span{\hbox{\rm span$\,$}}

\def\({\Big (}
\def\){\Big )}
\def\[{\Big[}
\def\]{\Big]}

\def\be{\begin{equation}}
\def\bel{\begin{equation}\label}
\def\ee{\end{equation}}
\def\bea{\begin{eqnarray}}
\def\eea{\end{eqnarray}}
\def\bt{\begin{theorem}}
\def\et{\end{theorem}}
\def\bc{\begin{corollary}}
\def\ec{\end{corollary}}
\def\bl{\begin{lemma}}
\def\el{\end{lemma}}
\def\bp{\begin{proposition}}
\def\ep{\end{proposition}}
\def\br{\begin{remark}}
\def\er{\end{remark}}
\def\ba{\begin{array}}
\def\ea{\end{array}}
\def\bd{\begin{definition}}
\def\ed{\end{definition}}

\makeatletter
   
   \@addtoreset{equation}{section}
\makeatother

\begin{document}

\title{\bf Hill-type formula for Hamiltonian system with Lagrangian boundary conditions}

\author[1]{Xijun Hu\thanks{Partially supported
by NSFC(No.11425105), E-mail:xjhu@sdu.edu.cn } }
\author[2]{Yuwei Ou\thanks{Partially supported
by NSFC(No.11671215,11571200), E-mail: ouyw3@mail.sysu.edu.cn} }
\author[1]{Penghui Wang\thanks{ Partially supported by NSFC(No.11471189),
 E-mail: phwang@sdu.edu.cn }}
\affil[1]{Department of Mathematics, Shandong University, Jinan, Shandong\authorcr
The People's Republic of China}
\affil[2]{School of Mathematics(Zhuhai),  Sun Yat-Sen University, Zhuhai, Guangdong\authorcr
 The People's Republic of China}


\maketitle
\begin{abstract}
In this paper, we build up Hill-type formula for linear Hamiltonian systems with Lagrangian boundary conditions, which include standard Neumann, Dirichlet boundary conditions. Such a kind of boundary conditions comes from the brake symmetry periodic orbits in $n$-body problem naturally.
The Hill-type formula connects the infinite determinant of the Hessian of the action functional with the determinant of  matrices which depend on the monodromy matrix and boundary conditions. Consequently, we  derive the Krein-type trace formula and give nontrivial estimation for the eigenvalue problem.
Combined with the Maslov-type index theory, we give some new  stability criteria for the brake symmetry periodic  solutions of Hamiltonian systems.  As an  application, we study the linear stability of  elliptic relative equilibria  in planar $3$-body problem.
 \end{abstract}

\bs

\no{\bf 2010 Mathematics Subject Classification}: 34L15,  34B09,  37C75, 70H14

\bs

\no{\bf Key Words}. Hill-type formula; Trace formula;  Conditional Fredholm determinant; Linear stability; Planar $3$-body problem

\section{Introduction}

In the present paper, we will study the eigenvalue problems of Hamiltonian systems with Lagrangian boundary conditions. Let $J=\left(\begin{array}{cc}0&-I_n\\
                                I_n&0\end{array}\right)$, the standard symplectic structure $\omega(x,y)$ on $\mathbb R^{2n}$ is defined by
$$
\omega(x,y)=\langle Jx,y\rangle.
$$
A Lagrangian subspace $V$ of $(\mathbb R^{2n},\omega)$ is an isotropic subspace of dimension $n$, that is,  for any $x,y\in V$, $\omega(x,y)=0$. Denote by $Lag(2n)$ the set of Lagrangian subspaces of $\mathbb R^{2n}$. For $V_0,V_1\in Lag(2n)$, the eigenvalue problem of linear Hamiltonian system with Lagrangian boundary condition   is to find $\lambda\in\mathbb C$ satisfying that the  nontrivial solutions of the following system exist,

\begin{eqnarray}
&&\dot{z}(t)= J(B(t)+\lambda D(t))z(t), \label{1.1e}\\
&&x(0)\in V_0,\quad x(T)\in V_1,\label{1.1f}
\end{eqnarray}
where $B,D\in C([0,T];\mathcal{S}(2n))$, the set of continuous paths of symmetric matrices. $B$ and  $D$ can be considered as bounded operators on $\cH:=L^2([0,T];\mathbb C^{2n})$, defined by $(Bx)(t)=B(t)x(t)$ and $(Dx)(t)=D(t)x(t)$.



  Unlike  the periodic boundary conditions, the Lagrangian boundary conditions are separated boundary conditions which include the Dirichlet, Neumann. However, periodic boundary conditions are closely related to Lagrangian boundary conditions. In fact, we will see that, for a periodic solution with brake symmetry, it is natural to obtain the Lagrangian boundary condition. Readers are referred to \cite{HS} and references therein to find the background of brake symmetry in $n$-body problem.

Denote by $A|_E=-J\frac{d}{dt}$ which is  densely defined on $\cH$ with the domain $E$. Obviously, the property of $A|_{E}$ depends on its domain $E$ heavily. The choice of $E$ is based on the boundary condition. For instance,  when we consider the Lagrangian boundary problem (\ref{1.1e})-(\ref{1.1f}) ,  the domain $E_{V_0,V_1}$ is defined as
$$
E_{V_0,V_1}=\left\{\left.z(t)\in W^{1,2}([0,T];\mathbb C^{2n})\,\right|\,z(0)\in V_0,\,  z(T)\in V_1\right\}.
$$

 It is well known that for $\lambda\in\rho(A|_E)$, the resolvent $(A|_E-\lambda)^{-1}$ is not a trace class operator, but a Hilbert-Schmidt operator.  Generally, in the case that $A|_{E}-B$ is non-degenerate, to simplify the notation,  we set
\begin{eqnarray}
\mathcal{F}(B,D;E)=D(A|_E-B)^{-1},\label{short} \nonumber
\end{eqnarray}
and  without of confusion, we will write \begin{eqnarray}A=A|_{E_{V_0,V_1}},\quad \mathcal{F}(B,D)(=\cF)=\mathcal{F}(B,D;E_{V_0,V_1})\label{shortA}.\nonumber \end{eqnarray}
Similar to the resolvent $(A|_E-\lambda)^{-1}$, in general  $\mathcal{F}(B,D;E)$ is not a trace class operator, but  a Hilbert-Schmidt operator.
It follows that the Fredholm determinant $\det(I-\cF(B,D;E))$ is not well-defined, instead we will use the definition of conditional Fredholm determinant, which was introduced in \cite{HW}. Some details will be recalled in Section 2.1. It is worth to be pointed out that there is another way to define the infinite dimensional determinant, which is defined by zeta function \cite{RaS,For}.

The first study of Hill formula was given by G.Hill in \cite{Hi} when he studied the motion of of lunar perigee, the strict mathematical proof of Hill's formula was given by H. Poincar\'{e} \cite{Po}.  There are many efforts on Hill-type formula were done, such as \cite{B,BT,Dr,HW,HW1,HW16} and references therein.
To state Hill-type formula, we
need some notations.   Suppose $\Lambda\in Lag(2n)$, a Lagrangian
frame for $\Lambda$ is a linear map $Z: \mathbb R^n \rightarrow
\mathbb R^{2n}$ whose image is $\Lambda$. It is easy to see that the
frame is of the form $Z=\left(\begin{array}{cc}X
\\ Y \end{array}\right)$, where $X,Y$ are $n\times n$ matrices and
satisfied $$ X^TY=Y^TX .$$  Denote  $\gamma_\lambda(t)$  the
fundamental solution of (\ref{1.1e}), that is
$\dot{\gamma}_\lambda(t)=J(B+\lambda D)(t)\gamma_\lambda(t)$ with
$\gamma_\lambda(0)=I_{2n}$.   Let  $$\Sp(2n):=\{M\in GL(\R^{2n})| M^TJM=J\}$$ be the symplectic group, it is well known that $\gamma(t)\in\Sp(2n)$.
Let $Z_0,Z_1$ be frames
of $\Lambda_0,\Lambda_1$. Obviously, $\gamma_\lambda(T) Z_0$ are
frames of $\ga_\lambda(T) \Lambda_0 $ and $(\ga_\lambda(T) Z_0,Z_1)$
are   $2n\times 2n$ matrices. We have the following Hill-type formula for  Hamiltonian system (\ref{1.1e})-(\ref{1.1f}).

\begin{thm}\label{thm1.1}
Assume that $A-B$ is non-degenerate, then
\begin{eqnarray}\label{t.1a1}
\det(I-\cF(B,D))=\det (\ga_{1}(T) Z_0,Z_1) \cdot
\det(\ga_0(T) Z_0,Z_1)^{-1},     \end{eqnarray}
where the left side
is the conditional Fredholm determinant, and the right side is independent on the choice of the frames $Z_0, Z_1$.
\end{thm}

\begin{rem}
 In \cite{HW16}, Hu and Wang  obtained  Hill-type formula for Sturm-Liouville system with separated boundary conditions.
It should be pointed out that, there is essential difference between \cite[Theorem 1.1]{HW16} and Theorem \ref{thm1.1}. More precisely, the differential operators in \cite[Theorem 1.1]{HW16} has trace class resolvent, and hence the classical Fredholm determinant can be defined. Therefore, in \cite{HW16} the techniques from complex analysis can be used, however such techniques do not work well here. Theorem \ref{thm1.1} will be proved by using the theory of integral operators.
\end{rem}

As mentioned above, by brake symmetry, solutions of Lagrangian boundary problem is  closely  related to  that of $S$-periodic boundary problem, where $S\in \Sp(2n)\cap O(2n)$ (the symplectic orthogonal group).
More precisely, we consider the following eigenvalue problem with  $S$-periodic conditions
\begin{eqnarray} \dot{z}(t)= J(B(t)+\lambda D(t))z(t), \quad  z(0)=Sz(T).\label{4.1}\nonumber
\end{eqnarray}
In \cite{HW}, the following Hill-type formula for $S$-boundary conditions was obtained
 \bea \det(I-\cF(B,D; E_S))=\det (S\ga_{1}(T)-I) \cdot
\det(S\ga_0(T)-I)^{-1},   \label{hilS}\nonumber \eea
where  $$E_S=\{z\in W^{1,2}([0.T],\R^{2n})| z(0)=Sz(T) \}.$$
Suppose there exists $N\in O(2n)$ such that $N^2=I_{2n}$,  $NJ=-JN$
 and $NS^T=SN$.   Then  we define
 \bea g:E\to E,\quad z(t)\mapsto Nz(T-t) \nonumber \eea
which generates a $\mathbb Z_2$ group action on $E_S$.  Obviously, $A|_{E_S}g=gA|_{E_S}$.  Suppose
\bea NB(T-t)=B(t)N,\quad  ND(T-t)=D(t)N, \nonumber  \eea  then $Bg=gB$.
Therefore
\bea (A|_{E_S}-B-\lambda D)g=g(A|_{E_S}-B-\lambda D),\quad for \quad \lambda\in\R.  \lb{4.4}\nonumber \eea
Let $V^+(SN), V^+(N)$ and $V^-(SN), V^-(N)$ be the positive and negative definite subspaces of $SN$ and $N$ respectively, then  $V^\pm(SN)$ and $V^\pm(N)$ are all Lagrangian subspaces of $(\R^{2n},\omega)$.

 Let
\bea \bar{E}^\pm_S=\{z\in E_S, gz=\pm z   \},\label{barespm} \eea  which are isomorphic to
  \bea E^\pm_S=\{z\in W^{1,2}([0,T/2], \mathbb C^{2n}), z(0)\in V^\pm(SN), z(T/2)\in V^\pm(N)   \}.   \lb{espm}\eea
We have the following decomposition formula, which build the relationship between the Hill-type formula of $S$-periodic boundary problem and  that of Lagrangian boundary problem.

\begin{thm}
Under the above conditions, we have \bea \det(I-\cF(B,D;E_S))=\det(I-\cF(B,D;E^+_S)\cdot \det(I-\cF(B,D;E^-_S)).  \label{hilde}\nonumber\eea
\end{thm}

By the similar idea to \cite{HOW}, using $\lambda \cF(B,D)$ instead of $\cF(B,D)$, and taking Taylor expansion on $\lambda$ for  both sides of Hill-type formula (\ref{t.1a1}), we have the trace formula. Trace formula is a powerful tool in studying the eigenvalue problem, and hence it is very useful to study the stability problem.  The  trace formula was first established by Krein in 1950's \cite{K1,K2}.  Recently, Hu and Wang give  the generalization to the   Sturm-Liuville system with general separated boundary conditions  \cite{HW}.  In this paper, we will build up the  Krein-type trace formula for Hamiltonian systems with Lagrangian boundary condition.  Please refer subsection 3.2 for the detail.

The motivation of Krein's trace formula was to study the stability problem of periodic orbits.
 Based on trace formula for $S$-periodic orbits and  Maslov index theory \cite{Lon4}, we give some new stability criteria and apply it to study the stability of $n$-body problem \cite{HOW}.
  As  continuous study,  we   use the trace formula for Hamiltonian system with Lagrangian boundary conditions to obtain  stability criteria for the brake symmetry periodic orbits. Consequently, we  give some applications on the study of planar $3$-body problem.

It  is well known that a planar central configuration of the n-body problem gives rise to solutions where each particle moves on a specific Keplerian orbit while the totality of the particles move on a homographic motion. Follows Meyer and Schmidt  \cite{MS},  we call this solution \emph{ elliptic relative equilibria} (ERE for short) if the Keplerian orbit is elliptic.    For  $n=3$, it is well known that  there are only two kind of  central configurations, the  Lagrangian equilateral triangular central configuration  and  the Euler collinear central configurations. We call the corresponding ERE are elliptic Lagrangian orbits and elliptic Euler orbits because they are first discovered by Lagrange  \cite{Lag} and Euler  \cite{Eu}.

There are many works in the study of linear stability of  elliptic Lagrangian orbits and elliptic Euler orbits.  Please refer to \cite{HLS,HOW,MS,MSS1,MSS2,LZhou,R1}  and references therein for the details.  More precisely,  let $e\in[0,1)$ be the eccentricity of the homothety Keplerian orbits of ERE.  The  linear stability of elliptic Lagrangian orbits  depends on $e$ and  $\beta\in[0,9]$ where  \be
\beta=\frac{27(m_1m_2+m_1m_3+m_2m_3)}{(m_1+m_2+m_3)^2}.\label{massratio}\nonumber\ee
Similarly,   the elliptic Euler orbits depend $e$ and $\delta$,  where  $\delta\in[0,7]$ only depends on masses $m_1, m_2, m_3$.

In \cite{HOW}, the first nontrivial estimation of the stability region and hyperbolic region of the $e,\beta$ rectangle $[0,1)\times[0,9]$ for elliptic Lagrangian orbits was given. In
Section 6, by observing  the elliptic Lagrangian orbits with brake symmetry, and we  give a  better estimation on the stability region by using the trace formula for Hamiltonian system with
Lagrangian boundary conditions. Moreover, the first nontrivial estimation of the hyperbolic region for elliptic Euler orbits will be given.

The paper is organized as follows. Section 2 is devoted to preliminaries on conditional Fredholm determinant and conditional trace. In Section 3, we will prove the Hill-type formula by using techniques in integral operators and complex analysis. Based on the Hill-type formula, we get the trace formula.  In Section 4 deal with the brake symmetry decomposition for Hill formula.  We  give some new stability criteria by  the trace formula in Section 5. At last, in Section 6, we give new estimation for  the stability region of elliptic Lagrangian orbits and estimation for hyperbolic region of  elliptic Euler orbits.

\section{Preliminaries}

In this section, we mainly introduce some fundamental notations and results which will be used later.  In Subsection \ref{sec2.1},  we give an overview on  conditional Fredholm determinant, details could be found in \cite{HW}.
 In Subsection \ref{sec2.2}, we compute the conditional trace  $\mathcal{F}(B,D)$.

\subsection{Conditional Fredholm determinant}\label{sec2.1}
In this subsection, we will mainly consider the conditional Fredholm determinant. As we have seen, the conditional Fredholm determinants of dynamical systems are our starting point to derive our trace formula.

Let $\cal {J}_\infty$ denote the family of compact operators. For $F\in \cal {J}_\infty$, let
$\mu_1\geq\mu_2\geq\cdots\geq$ be the singular values of $A$. In fact,  $\mu_j$ are eigenvalue of $|F|:=(F^*F)^{\frac{1}{2}}$.  For $p\geq1$, we denote $$ \cal {J}_p:=\{F\in \cal {J}_\infty | \sum \mu_n(F)^p<\infty \},  $$
and $ \|F\|_p:=(\sum \mu_n(F)^p)^{\frac{1}{p}} $ for $F\in  \cal {J}_p$.  Obviously, $\cal {J}_1 $ is the set of trace class operator and $\cal {J}_2$ is the set of  Hilbert-Schmidt operators. It is well known that  for $F\in\cal {J}_1 $,  the Fredholm determinant $\det(id+F)$ is well defined. However, for $F\in\cal{J}_2\setminus \cal{J}_1$, such a Fredholm determinant can not be well-defined.  Instead, the  regularized determinant
 \begin{eqnarray*}
{\det}_2 (id+F)=\det\left((id+F)e^{-F}\right)
\end{eqnarray*}
is  well-defined since $(id+F)e^{-F}-id\in \cal{J}_1$.  It is known that the regularized determinant has no multiplicative  property. Thus we hope to define a kind of conditional Fredholm determinant for $(id+F)$.

The concept of \emph{trace finite condition} plays an important role in the study of conditional Fredholm determinant.
Firstly, we will recall the definition of trace finite condition, which is introduced in \cite{HW}.
Let $\{P_k\}$ be a sequence of finite rank projections, such that the following conditions are satisfied,
\begin{itemize}
\item[(1)] for $k\leq m$, $Range(P_k)\subseteq Range (P_m)$,
\item[(2)] $P_k$ converges to $id$ in the strong operator topology.
\end{itemize}
 We denote \bea \cal{J}(P_k):=\{F\in\cal{J}_2 | \lim\limits_{k\to\infty} Tr(P_k F P_k) \, exists  \,  and \,is\, finite  \}\nonumber \eea
be the set of  operators with trace finite condition respect to $\{P_k\}$.

 It is obvious that $\cal{J}(P_k)$ is linear space and
 \bea  \cal{J}_1\subset \cal{J}(P_k)\subset \cal{J}_2. \nonumber \eea
 As been pointed in \cite{HW1}, if $F\in\cal{J}(P_k)$, then the conditional Fredholm determinant can be defined,
\begin{eqnarray}
\det(id+F)&=&\lim\limits_{k\to\infty} \det(id+P_k F P_k)\nonumber\\
          &=&\lim\limits_{k\to\infty} {\det}_2(id+P_k F P_k) e^{Tr(P_k F P_k)}\nonumber\\
          &=&{\det}_2 (id+F) \lim\limits_{k\to\infty} e^{Tr(P_k F P_k)}.\label{eq1.5a}
\end{eqnarray}
By \cite{HW1}, if $F\in\cal{J}_1$, then conditional Fredholm determinant $\det(id+F) $ is the
classical Fredholm determinant.  Moreover for $F\in\cal{J_1}$, the function $\det(id+\alpha F)$ is analytic on $\alpha$. Correspondingly,  in \cite[Lemma 3.3]{HW1}, we proved that the conditional Fredholm determinant $\det\left((A-B-\alpha D-\nu J)(A+P_0)^{-1}\right)$ is analytic on $\alpha$. Similarly, it is not hard to see that for $F\in{\cal{J}(P_k)}$, the function $\det(id+\alpha F)$, defined by conditional Fredholm determinant,  is an entire function.

At the end of this subsection, we will list some fundamental properties  of the conditional Fredholm determinant, which were proved in \cite{HW}.
\begin{prop}\label{prop2.5}
\begin{itemize}
 \item[1)] If $F_1, F_2\in \cal{J}(P_k)$, then
\begin{eqnarray*}
\det\left((id+F_1)(id+F_2)\right)=\det(id+F_1)\det(id+F_2).
\end{eqnarray*}
\item[2)] Let $E=E_1\oplus E_2$, and  $F_i\in \cal{J}(P_k^{(i)})$ for  $i=1,2$. Let   $F=F_1\oplus F_2$, then $F\in\cal{J}(P_k^{(1)}\oplus P_{k}^{(2)} )  $, and
\begin{eqnarray*}
\det(id+F)=\det(id_{E_1}+F_1)\det(id_{E_2}+F_2),
\end{eqnarray*}
where $id_{E_i}$ are identities on $E_i$, for $i=1,2$.
\end{itemize}
\end{prop}

\subsection{Conditional trace for operator with Lagrangian subspace boundary conditions}\label{sec2.2}

Suppose $V_0,V_1\in Lag({2n})$.
By Changing  a symplectic basis, we  may assume  $V_0, V_1$ with Lagrangian frames \begin{eqnarray}\label{Z-0}Z_0=(I_n,0_n)^T,\quad Z_1=(C(\theta),S(\theta))^T,\nonumber\end{eqnarray}  where for $-\pi/2<\theta_j\leq \pi/2$,
\begin{eqnarray}\label{Note:CS}
C(\theta)=diag(\cos(\theta_1),\cdots,\cos(\theta_n)), \quad  S(\theta)=diag(\sin(\theta_1),\cdots,\sin(\theta_n)).\nonumber
\end{eqnarray}
Then $A$ is a self-adjoint
operator on $\cH$ with domain $E(V_0,V_1)$; moreover,  $A$ has
compact resolvent and only has point spectrum.  Easy computation shows
$$
\sigma_p\big(A)=\{\lambda_{j,k}\,|\, \lambda_{j,k}=\theta_j/T+k\pi/T, j=1,2\cdots,n, \,k\in \mathbb Z\},
$$
with the corresponding eigenvectors $e_{j,k}=e^{\lambda_{j,k}Jt}e_j=\cos(\lambda_{j,k}t)e_j+\sin(\lambda_{j,k}t)e_{n+j}$, where $e_j$ is the standard $j$-th basis of $\mathbb{R}^{2n}$.  In what follows,
let $P_N$ be the projections from $\cH$ to $\span\{e_{j,k};\, |k|\leq N\}$.

\begin{prop}\label{cor:trace}
(i) For any $\nu\in\mathbb{C}$ such that $A-\nu$ is invertible,  $\cF(\nu,D)\in\cal{J}(P_N)$  and  \bea  Tr \cF(\nu,D)=\sum_{j=1}^n \frac{\cos (\theta_j-T\nu)}{\sin
(\theta_j-T\nu)}\int_0^T(e^{-\nu Jt} De^{\nu Jt}e_j,e_j)dt +\sum_{j=1}^n\int_0^T(e^{-\nu Jt}D e^{\nu Jt} e_{n+j},e_j)dt. \label{tr.1}\nonumber \eea \\
(ii) $\cF(B,D)\in\cal{J}(P_N)$ and hence the conditional Fredholm determinant $\det (I-\cF(B,D))$ is well defined.
\end{prop}

Since $\cF(\nu,D)-\cF(B,D)=\cF(\nu,D)(\nu-B)(A-B)^{-1}\in\cal{J}_1$, part ii) follows from part i) easily. The remaining part of this subsection is devote to the proof of part i), which is technical. Readers may skip it on the first glance.

Set
$\hat{D}=\frac{1}{2}(D-JDJ)$, $\check{D}=\frac{1}{2}(D+JDJ)$, then
obviously $D=\hat{D}+\check{D}$, and $ \hat{D}J=J\hat{D},\,\
\check{D}J=-J\check{D}$.  Therefore \bea  e^{Jt}\hat{D}=\hat{D}e^{Jt}, \quad e^{Jt}\check{D}=\check{D}e^{-Jt}. \label{ejd}  \nonumber\eea
Obviously,
\bea   Tr  \cF(\nu,D)= Tr  \cF(\nu,\hat{D}) +Tr \cF(\nu, \check{D}),\nonumber\eea
and $Tr \cF(\nu,\hat{D})$  and $Tr \cF(\nu,\check{D})$ will be computed  separately.
Throughout  the paper, in order to simplify the notations, the summation  $\sum_{k}$ always means $\lim_{N\rightarrow\infty}\sum_{|k|\leq N} $.
\begin{lem}\label{lem2.2} For any $\nu\in\mathbb{C}$ such that $A-\nu$ is invertible,  \bea Tr \cF(\nu,\hat{D})=\sum_{j=1}^n \frac{\cos (\theta_j-T\nu)}{\sin
(\theta_j-T\nu)}\int_0^T(\hat{D}e_j,e_j)dt , \label{p2.2f1}\nonumber\eea
\end{lem}
\begin{proof}
By the definition \bea
Tr \cF(\nu,\hat{D})&=&\frac{1}{T}\sum_{j=1}^n \sum_{k}\langle\hat{D}e^{\lambda_{j,k}Jt}e_j,e^{\lambda_{j,k}Jt}e_j\rangle\frac{1}{\lambda_{j,k}-\nu}
\nonumber \\ &=&\frac{1}{T}\sum_{j=1}^n \int_0^T\langle\hat{D}e_j,e_j\rangle dt \sum_{k}\frac{1}{k\pi/T+\theta_j/T-\nu} \nonumber \\ &=& \sum_{j=1}^n \frac{\cos (\theta_j-T\nu)}{\sin
(\theta_j-T\nu)}\int_0^T(\hat{D}e_j,e_j)dt ,\nonumber\eea where the third
equality is from the the following identity \cite[Lemma 2.7]{HW1}
$$\lim_{N\rightarrow\infty}\sum_{|k|\leq
N}\frac{1}{\nu+k\pi/T}=T\frac{\cos (T\nu)}{\sin
(T\nu)}.$$
\end{proof}

The following lemma is needed to compute $Tr \cF(\nu,\check{D}) $.
\begin{lem}\label{lem2.3} For $f\in C[0,T]$, we have   \bea \int_0^Tf(t)\sum_k \frac{\cos 2(k\pi/T+\theta_j/T-\nu)t}{k\pi/T+\theta_j/T-\nu}dt= T\int_0^Tf(t)dt\frac{\cos (\theta_j-\nu T)}{\sin (\theta_j-\nu T)},
\label{f2.1}\eea
\bea \int_0^Tf(t)\sum_k \frac{\sin 2(k\pi/T+\theta_j/T-\nu)t}{k\pi/T+\theta_j/T-\nu}= T\int_0^Tf(t)dt. \label{f2.2}\eea

\end{lem}
\begin{proof}
For $f\in C[0,T]$, set $F(t)=\int_0^tf(s)ds$. Notice that $$\sum\limits_{|k|\leq N}\frac{\cos
(2(\lambda_{j,k}-\nu)t)}{\lambda_{j,k}-\nu}=\frac{\cos
(2({\theta_j/T}-\nu)t)}{\theta_j/T-\nu}+\sum\limits_{k=1}^N\(\frac{\cos
(2(\lambda_{j,k}-\nu)t)}{\lambda_{j,k}-\nu}+\frac{\cos
(2(\lambda_{j,-k}-\nu)t)}{\lambda_{j,-k}-\nu}\),$$ which converges uniformly. By integration by part, we have
\bea
&&\int_0^Tf\cdot\sum_k\frac{\cos
(2(\lambda_{j,k}-\nu)t)}{\lambda_{j,k}-\nu} dt\nonumber\\&=&
F(T)\sum_k\frac{\cos
(2(\lambda_{j,k}-\nu)T)}{\lambda_{j,k}-\nu}+2\int_0^TF(t)\sum_k\sin
(2(\lambda_{j,k}-\nu)t)dt. \label{eq:2.10} \eea
For the first part of the above equality
\bea \sum_k\frac{\cos(
2(\lambda_{j,k}-\nu)T)}{\lambda_{j,k}-\nu}&=&\sum_k\frac{\cos(2\theta_j-2\nu
T)}{k\pi/T+\theta_j/T-\nu } \nonumber \\ &=& T\cos(2\theta_j-2\nu
T)\frac{1+\cos(2\theta_j-2\nu
T)}{\sin(2\theta_j-2\nu T)}, \nonumber\eea
For the second part of \eqref{eq:2.10}
\bea
&&\int_0^TF(t)\sum_k\sin (2(\lambda_{j,k}-\nu)t)dt \nonumber
\\&=&\int_0^TF(t)\sin((2\theta_j/T-2\nu)t)\sum_k\cos
\frac{2k\pi t}{T}dt+\int_0^TF(t)\cos((2\theta_j/T-2\nu)t)\sum_k\sin
\frac{2k\pi t}{T}dt \nonumber \\ &=&2
\int_0^TF(t)\sin((2\theta_j/T-2\nu)t)\sum_{k\in\mathbb{N}}\cos
\frac{2k\pi t}{T}dt+ \int_0^TF(t)\sin((2\theta_j/T-2\nu)t)dt \nonumber            \\
&=&\frac{T}{2}F(T)\sin(2\theta_j-2\nu T),\nonumber            \\
&=& \frac{T}{2}\int_0^Tf(t)dt\cdot \sin(2\theta_j-2\nu T),\nonumber\eea where the
third equality is from the Fejer Theorem for Fourier series. Thus we
have
 \bea &&\int_0^Tf(t)\sum_k\frac{\cos
(2(\lambda_{j,k}-\nu)t)}{\lambda_{j,k}-\nu} dt\nonumber\\ &=&
T\int_0^Tf(t)dt\cdot[\cos(2\theta_j-2\nu T)\frac{1+\cos(2\theta_j-2\nu
T)}{\sin(2\theta_j-2\nu T)}+ \sin(2\theta_j-2\nu T)]\nonumber \\ &=&
T\int_0^Tf(t)dt\cdot\frac{1+\cos(2\theta_j-2\nu T)}{\sin(2\theta_j-2\nu
T)}.\nonumber \eea

Similarly \bea &&\int_0^Tf\cdot\sum_k\frac{\sin
(2(\lambda_{j,k}-\nu)t)}{\lambda_{j,k}-\nu} dt\nonumber\\ &=&
\int_0^Tf(t)dt\sum_k\frac{\sin
(2(\lambda_{j,k}-\nu)T)}{\lambda_{j,k}-\nu}-2\int_0^TF(t)\sum_k\cos(
2(\lambda_{j,k}-\nu)t)dt, \nonumber \eea
and
 \bea \sum_k\frac{\sin
(2(\lambda_{j,k}-\nu)T)}{\lambda_{j,k}-\nu}&=&\sum_k\frac{\sin(2\theta_j-2\nu
T)}{k\pi/T+\theta_j/T-\nu } \nonumber  = T(1+\cos(2\theta_j-2\nu
T)). \eea
As above, we have
\bea \int_0^TF(t)\sum_k\cos
(2(\lambda_{j,k}-\nu)t)dt&=&\int_0^TF(t)\cos((2\theta_j/T-2\nu)t)\sum_k\cos
\frac{2k\pi t}{T}dt\nonumber \\
&&-\int_0^TF(t)\sin((2\theta_j/T-2\nu)t)\sum_k\sin \frac{2k\pi
t}{T}dt \nonumber
\\ &=& \frac{T}{2}(F(0)+F(T)\cos(2\theta_j-2\nu T)) \nonumber
\\ &=& \frac{T}{2}\int_0^Tf(t)dt\cos(2\theta_j-2\nu T). \nonumber\eea
We have \bea \int_0^Tf\cdot\sum_k\frac{\sin
(2(\lambda_{j,k}-\nu)t)}{\lambda_{j,k}-\nu} dt =T\int_0^Tf(t) dt.\nonumber
\eea

\end{proof}

\begin{lem}\label{lem2.3} For any $\nu\in\mathbb{C}$ such that $A-\nu$ is invertible,  then $\cF(\nu,\check{D})\in\cal{J}(P_N)$  and the conditional trace
 \bea &&Tr \cF(\nu,\check{D})= \sum_{j=1}^n \frac{\cos (\theta_j-T\nu)}{\sin
(\theta_j-T\nu)}\int_0^T(e^{-\nu Jt}\check{D}e^{-\nu Jt}e_j,e_j)dt +\sum_{j=1}^n\int_0^T(e^{-\nu Jt}\check{D}e^{\nu Jt}e_{n+j},e_j)dt.
\label{le2.4.1}\nonumber\eea
\end{lem} 
\begin{proof}
By the definition \bea
Tr \cF(\nu,\check{D})&=&\frac{1}{T}\sum_{j=1}^n\sum_{k}\langle\check{D}e^{\lambda_{j,k}Jt}e_j,e^{\lambda_{j,k}Jt}e_j\rangle\frac{1}{\lambda_{j,k}-\nu}
\nonumber \\ &=&\frac{1}{T}\sum_{j=1}^n\sum_{k}\int_0^T( e^{-2\nu Jt}\check{D}\frac{e^{2(k\pi/T+\theta_j/T-\nu)Jt}}{k\pi/T+\theta_j/T-\nu}    e_j,e_j)dt. \label{l2.3.1}
\eea
Please note that \bea\int_0^T( e^{-2\nu Jt}\check{D}\frac{e^{2(k\pi/T+\theta_j/T-\nu)Jt}}{k\pi/T+\theta_j/T-\nu}    e_j,e_j)dt&=&
  \int_0^T(e^{-2\nu Jt}\check{D}e_j,e_j)\frac{\cos [2(k\pi/T+\theta_j/T-\nu)t ]}{k\pi/T+\theta_j/T-\nu }dt\nonumber\\&&+\int_0^T(e^{-2\nu Jt}\check{D}e_{n+j},e_j) \frac{\sin [2(k\pi/T+\theta_j/T-\nu)t ]}{k\pi/T+\theta_j/T-\nu}dt.   \nonumber\eea
By (\ref{f2.1}-\ref{f2.2}) and $e^{-\nu Jt}\check{D}=\check{D} e^{Jt}$,  the right of (\ref{l2.3.1}) equals to
\bea \sum_{j=1}^n \frac{\cos (\theta_j-T\nu)}{\sin
(\theta_j-T\nu)}\int_0^T(e^{-2\nu Jt}\check{D}e_j,e_j)dt +\sum_{j=1}^n\int_0^T(e^{-\nu Jt}\check{D}e^{\nu Jt}e_{n+j},e_j)dt,  \nonumber\eea
this is end of the proof.
\end{proof}

Please note that  $\sum_{j=1}^n\int_0^T(e^{-\nu Jt}\hat{D} e^{\nu Jt} e_{n+j},e_j)dt=0$.    Part i) of Proposition \ref{cor:trace} comes from Lemma \ref{lem2.2} and Lemma \ref{lem2.3} directly.

\section{Hill-type formula for Hamiltonian systems with  Lagrangian boundary conditions }

This section is the main part of our paper,  we build up the Hill-type formula in Subsection \ref{sec3.1}, and the trace formula is obtained in Subsection \ref{sec3.2}.  At last, we discuss the relationship between the the eigenvalue problem for Hamiltonian system  with  that of  Sturm-Liouville systems in Section \ref{sec3.3}. 
\subsection{Hill-type formula}\label{sec3.1}
Let $\gamma_\alpha$ be the fundamental solution of $B+\alpha D$, that is
\bea  \dot{\gamma}_\alpha(t)=J(B(t)+\alpha D(t))\gamma_\alpha(t), \quad \gamma(0)=I_{2n}.\nonumber  \eea
Assume $A-B$ is nondegenerate which is obvious equivalent to $V_0\pitchfork \gamma^{-1}_0(T)V_1$.
We will express $\cF(B,D)$ by  integral operator.   Let $Q$ be the unique  idempotent  matrix on $\mathbb{R}^{2n}$ with kernel  $V_0$ and image $\gamma_0^{-1}(T)V_1$, in general  $Q$ is not orthogonal.  The  integral kernel of  $\cF(B,D)$
could be given by
\bea \cal{ K}_\cF (t,t')=\left\{\begin{array}{cc} JD\gamma_0(t)Q\gamma^{-1}_0(t'), \quad 0\leq t'<t;
\\ -JD\gamma_0(t)(I_{2n}-Q)\gamma^{-1}_0(t'), \quad t\leq t'<T.   \end{array}\right.\nonumber\eea
That means $$\cF(B,D) f(t)=\int_0^T \cal{K}_{\cF}(t,t^\prime)f(t')dt'.$$
In general, the kernel $\cal{K}_\cF$ is not continuous but in $L^2$. However,  we may  define the trace formally from the viewpoint of integral operators
\begin{defi}\label{Defi-Trace} For $\cal {K}_{\cF}$, we define
\bea Tr \cal{K}_\cF =Tr\int_0^T JD\gamma_0(t)Q\gamma^{-1}_0(t)dt.  \nonumber \eea
\end{defi}

At first, we will show the following lemma.

\begin{lem}\label{lemma3.2}  For any $\beta\in[0,1]$,
\bea  Tr \cal{K}_\cF =Tr \int_0^T [\beta JD\gamma_0(t)Q\gamma^{-1}_0(t)-(1-\beta)JD\gamma_0(t)(I_{2n}-Q)\gamma^{-1}_0(t)]  dt. \nonumber  \eea
\end{lem}
\begin{proof}
Since $D$ is a path of symmetric matrices, it follows that
$$
Tr\int_0^T -JDdt=0,
$$
Hence
$$
Tr\int_0^T JD \gamma_0(t)Q\gamma_0^{-1}(t)dt=-Tr\int_0^TJD\gamma_0(t)(I_{2n}-Q)\gamma_0^{-1}(t)dt.
$$
The proof is complete.
\end{proof}

To simply the notation, we let $$Q_d=\left(\begin{array}{cc}I_n & 0_n \\0_n  & 0_n \end{array}\right), Q_n=\left(\begin{array}{cc}0_n & 0_n \\0_n  & I_n \end{array}\right).     $$  In the following Lemma \ref{l3.3} and Corollary \ref{c3.5}, we will show that the conditional trace defined in Subsection 2.2 coincides with the formally defined trace in Definition \ref{Defi-Trace}.
\begin{lem}\label{l3.3} For the case $B=\nu I_{2n}$,
\bea  Tr \cF(\nu,D) = Tr \cal{K}_{\cF} . \nonumber  \eea
\end{lem}
\begin{proof} In this case $B=\nu I_{2n}$, $\gamma_0(t)=e^{\nu Jt}$.  The frame of $\gamma_0^{-1}(T)V_1$ could be given by
$(C(\theta-\nu T ), S(\theta-\nu T))^T$. In what follows, we will set \begin{eqnarray}P:=(Z_0, \gamma^{-1}_0(T)Z_1 )=\left(\begin{array}{cc}I_n & C(\theta-\nu T) \\0_n  & S(\theta-\nu T) \end{array}\right).\nonumber  \end{eqnarray} Hence $$Q=PQ_n  P^{-1}=\left(\begin{array}{cc}0_n & C(\theta-\nu T)S(\theta-\nu T)^{-1}  \\0_n  & I_n \end{array}\right).$$
By definition, \bea Tr \cal{K}_\cF  &=& Tr \int_0^T JD\gamma_0(t)Q\gamma^{-1}_0(t)dt \nonumber \\
&=& Tr \int_0^T e^{-\nu Jt}De^{\nu Jt}QJ dt \nonumber \\  &=& \sum_{j=1}^n \frac{\cos (\theta-\nu T)}{\sin
(\theta-\nu T)}\int_0^T(e^{-\nu Jt} De^{\nu Jt}e_j,e_j)dt +\sum_{j=1}^n\int_0^T(e^{-\nu Jt}D e^{\nu Jt} e_{n+j},e_j)dt. \nonumber\eea
Combining with Corollary \ref{cor:trace}, we have the result.
\end{proof}

\begin{cor}\label{c3.5}  For any $B$ such that $A-B$ is invertible, we have
\bea Tr \cF(B,D)= Tr \cal{K}_\cF.  \label{cf3.1}\nonumber\eea\end{cor}
\begin{proof}
Please note that
$\cF(B,D)-\cF(\nu,D)\in\cal{J}_1$. We have, $\cal{K}_{\cF(B,D)}-\cal{K}_{\cF(\nu,D)}$ is the integral kernel of a
trace class operator. By \cite[P.244, Theorem 2.1]{GGK}, the kernel $\cal{K}_{\cF(B,D)}-\cal{K}_{\cF(\nu,D)}$ is continuous, and hence by \cite[P.70, Theorem 8.1]{GGK},
$$
Tr(\cF(B,D)-\cF(\nu,D))=Tr\int_0^T (\cal{K}_{\cF(B,D)}-\cal{K}_{\cF(\nu,D)})(s,s)ds=Tr(\cal{K}_{\cF(B,D)}-\cal{K}_{\cF(\nu,D)}).
$$
The result is from Lemma \ref{l3.3}.

\end{proof}

Let $M(\alpha)=(\gamma_\alpha(T)Z_0,Z_1)_{2n\times2n}$. In the remaining part of this section, we will let \bea \label{def:f}f(\alpha)=\det M(\alpha)\det M^{-1}(0). \label{falpha}\eea Direct computation shows that \bea \gamma^{-1}_0(T)M(\alpha)&=&(\gamma^{-1}_0(T)\gamma_\alpha(T)Z_0, \gamma^{-1}_0(T)Z_1 ) \nonumber \\
    &=& P+((\gamma^{-1}_0(T)\gamma_\alpha(T)-I_{2n})Z_0, 0_{2n\times n} )  \nonumber \\
    &=& P\(I_{2n}+P^{-1}(\gamma^{-1}_0(T)\gamma_\alpha(T)-I_{2n})  Q_d \). \nonumber \eea
From \cite{HOW},  we have  $\frac{d}{d\alpha} (\gamma^{-1}_0(T)\gamma_\alpha(T)-I_{2n})\mid_{\alpha=0} =\int_0^T\gamma^{-1}_0(t)JD\gamma_0(t)dt$. Then

\bea f'(0)&=& Tr  P^{-1} \int_0^T\gamma^{-1}_0(t)JD\gamma_0(t)dt Q_d \nonumber \\  &=& Tr   \int_0^T\gamma^{-1}_0(t)JD\gamma_0(t)dt (I-Q) , \nonumber \eea
where the second equality from the fact that $I-Q=PQ_dP^{-1}$, $PQ_d=Q_d$. By Lemma \ref{lemma3.2}, we have

\begin{lem} For any $B$ such that $A-B$ is invertible,
\bea Tr \cF(B,D)= Tr \cal{K}_\cF= -f'(0). \nonumber \eea\end{lem}

From the above discussion, we have

\begin{thm}\label{hill} For any $B$ such that $A-B$ is invertible,
\bea  \det (I-\alpha \cF(B,D))= f(\alpha). \label{hillf} \eea
\end{thm}
\begin{proof} Let $g(\alpha)= \det (I-\alpha \cF(B,D))$, then $f$ and $g$ are analytic functions on $\mathbb C$ with same zero points and $f(0)=g(0)=1$.   We will show that for $f(\alpha)\neq0$,
\begin{eqnarray}
\label{formula1}g'(\alpha)g^{-1}(\alpha)=f'(\alpha)f^{-1}(\alpha),
\end{eqnarray} which  implies (\ref{hillf}).  For any $\alpha_0$ with $f(\alpha_0)\neq0$, we have
$$ g(\alpha)= \det (I-(\alpha-\alpha_0)\cF(B-\alpha_0 D,D) g(\alpha_0),$$
so $g'(\alpha_0)g^{-1}(\alpha_0)= -Tr \cF(B-\alpha D,D)     $.  On the other hand,  $$f(\alpha)=\det M(\alpha)\det M^{-1}(\alpha_0) f(\alpha_0).$$ Using $B+\alpha D$ instead of $B$ in  Corollary  \ref{c3.5}, we get
\bea -Tr \cF(B-\alpha D,D)= \frac{d}{d\alpha}\det M(\alpha)\det M^{-1}(\alpha_0)|_{\alpha=\alpha_0}=f'(\alpha)f^{-1}(\alpha), \nonumber \eea
and hence the equality (\ref{formula1}) is proved.
\end{proof}

\begin{rem}
The Hill-type formula \eqref{hillf} shows that $f(\alpha)$  is independent on the choice of the frames of $V_0$ and $V_1$. In fact, we can get this by easy computation  of  \eqref{falpha}.
\end{rem}

\subsection{ Krein-type  trace formula  }\label{sec3.2}

Since $\cF=\cF(B,D)\in\cal{J}(P_N)$, we have
\bea\label{expansion1}
\det (I-\alpha \cF)=exp\left(-\sum\limits_{m=1}^\infty{1\over m}\alpha^m Tr ( \cF^m)\right)
\eea
Next, we will give the expansion of $f(\alpha)$.
Recall that  $P=(Z_0, \gamma^{-1}_0(T)Z_1)$, $Q_d=(Z_0,0_{2n\times n})$.
Noting that $M(0)=\gamma_{0}(T)(Z_0,\gamma_0^{-1}(T)Z_1)=\gamma_0(T)P$, and
$\det(\gamma_0(T))=1$, we have
\bea
f(\alpha)&=&\det\left(P\(I_{2n}+P^{-1}(\gamma^{-1}_0(T)\gamma_\alpha(T)-I_{2n}) Q_d \)\right)\det(P^{-1})\nonumber\\
         &=&\det\(I_{2n}+P^{-1}(\gamma^{-1}_0(T)\gamma_\alpha(T)-I_{2n}) Q_d \).\nonumber
\eea
Let $\hat{\gamma}_\alpha(T)=\gamma^{-1}_0(T)\gamma_\alpha(T)$, by \cite[Section 2.2]{HOW},
\bea
\hat{\gamma}_\alpha(T)-I_{2n}=\sum\limits_{j=1}^\infty \alpha^j M_j,\nonumber
\eea
where for $\hat{D}(t)=\gamma_{0}^T(t)D(t) \gamma_{0}(t)$,
\bea
M_j=\int_0^TJ\hat{D}(t_1)\int_0^{t_1}J\hat{D}(t_2)\cdots\int_0^{t_{j-1}}J\hat{D}(t_j)dt_j\cdots dt_2dt_1.\nonumber
\eea
Let
$
G_j=P^{-1} M_jQ_d,
$
we have that $$f(\alpha)=\det(I_{2n}+\sum\limits_{j=1}^\infty \alpha^j G_j).$$ Since $f(\alpha)$ vanishes nowhere near $0$, we can write $f(\alpha)=e^{g(\alpha)}$, then by \cite[Formula 2.6]{HOW}, we have
\bea\label{expansion2}
g^{(m)}(0)/m!=\sum\limits_{j=1}^m{(-1)^{k+1}\over k}\left(\sum\limits_{j_1+\cdots+j_k=m}Tr(G_{j_1}\cdots G_{j_k})\right).
\eea
Combining (\ref{expansion1}) and (\ref{expansion2}) with Theorem \ref{hill}, we have the following trace formula.
\begin{thm}
With the above notations, we have that
\begin{eqnarray}
Tr ( \cF^m)=m\sum\limits_{k=1}^m {(-1)^k\over k}\left(\sum\limits_{j_1+\cdots+j_k=m}Tr(G_{j_1}\cdots G_{j_k})\right).\nonumber
\end{eqnarray}
\end{thm}

For $m=1,2$, the trace formula is simple
\bea
Tr \cF=-Tr(G_1),\quad Tr ( \cF^2)=Tr(G_1^2)-2Tr(G_2). \lb{tracef}
\eea
Since for $m\geq2$, $\cF(B,D)$ is trace class operator, we have
\bea
Tr (\cF^m)=\sum_{j} \lambda_j^{-m} ,\nonumber
\eea
where each $\lambda_j$ appears as many times as its multiplicity.


\subsection{Relation with the eigenvalue problem of  Sturm-Liouville systems}\label{sec3.3}
In \cite{HW}, the Hill-type formula and Krein-type trace formula were given for Sturm-Liouville system. In this subsection, we will study the relationship between the formulas for general Hamiltonian systems and that for Sturm-Liouville systems.
When the Hamiltonian system comes from the Legendre transformation of Sturm-Liouville systems, the operator $\cF(B,D)\in\cal{J}_1$, then
\bea  \det(I-\cF)&=&\prod_{j}(1-\lambda_j^{-1}),  \lb{3.3.1} \\
Tr(\cF)&=&\sum_{j} \lambda_j^{-1},   \lb{3.3.2}  \eea
where $\lambda_j's$ are eigenvalues the corresponding Sturm-Liouville systems:
 \begin{eqnarray}
-(P\dot{y}+Qy)^\cdot+Q^T\dot{y}+(R+{\lambda} R_1)y=0, \label{sl1} \end{eqnarray}
where  $Q$ is a continuous path of $n\times n$ matrices, and $P, R, R_1$ are continuous paths of $n\times n$ symmetric matrices on $[0,T]$. Instead of Legendre convexity condition, we assume that for any $t\in [0,T]$,  $P(t)$  is  invertible.
The boundary condition given in follows:
 let $\Lambda_0, \Lambda_1\in Lag(2n)$,  which are  phase spaces with standard symplectic structure.  Set  $x=P\dot{y}+Qy$, $z=(x,y)^T$, and the boundary condition is given by
\bea z(0)\in\Lambda_0,\quad  z(T)\in\Lambda_1. \label{slb1} \eea

By the standard Legendre transformation,  the linear  system
(\ref{sl1}) with the boundary conditions (\ref{slb1}) corresponds to the linear Hamiltonian system,
 \bea \dot{z}=JB_\lambda(t)z, \quad z(0)\in\Lambda_0,\quad  z(T)\in\Lambda_1, \label{h1}\eea
with \bea  B_\lambda(t)=\left(\begin{array}{cc}P^{-1}(t)& -P^{-1}Q(t) \\
-Q(t)^TP^{-1}(t)  & Q(t)^TP^{-1}(t)Q(t)-R(t)-\bar{\lambda}  R_1(t)\nonumber
\end{array}\right).\label{b2} \eea
 We denote  $\gamma_\lambda(t)$  the
fundamental solution of (\ref{h1}). Let $Z_0,Z_1$ be frames
of $\Lambda_0,\Lambda_1$.
To simplify the notation, set $\cA=-\d(P\d+Q)+Q^T\d+R$, which is a
self-adjoint operator on $L^2([0,T], \mathbb R^{n})$ with domain:
$$   D(\Lambda_0,\Lambda_1)=\{y\in W^{2,2}([0,T],  \mathbb R^{n}),   z(0)\in\Lambda_0,   z(T)\in\Lambda_1  \}.      $$
 We  assume
$\cA$ is  nondegenerate,  that is, $0$ is not an eigenvalue of
(\ref{sl1}-\ref{slb1}). It is obvious that $\lambda$ is a nonzero
eigenvalue of the system(\ref{sl1}-\ref{slb1}) if and only if
$-{1\over \lambda}$ is an eigenvalue of $R_1\cA^{-1}$. In what follows, the multiplicity of an eigenvalue $\lambda_j$ means the algebraic multiplicity of $R_1\cA^{-1}$ at ${-1/\lambda_j}$.
It was proven in \cite{HW16} that
\begin{eqnarray}\label{3.28}
\prod_{j}(1-\lambda_j^{-1})=\det(I+R_1\cA^{-1})=\det (\ga_{1}(T) Z_0,Z_1) \cdot
\det(\ga_0(T) Z_0,Z_1)^{-1},     \end{eqnarray}
Let $D=diag(0_n, R_1)$, then Hill formula (\ref{hillf}) shows that
\bea \label{3.29} \det(I+\mathcal{F}(B_0,D))=\det (\ga_{1}(T) Z_0,Z_1)\cdot
\det(\ga_0(T) Z_0,Z_1)^{-1}, \eea
we have
  \begin{cor}\label{cor3.3.1} Under the above notations
 \begin{eqnarray}\label{3.30}
 \det(I+\mathcal{F}(B_0,D))=\det(I+R_1\cA^{-1}),   \end{eqnarray}
consequently,  $\sigma(\cF(B_0,D))=\sigma(R_1\cA^{-1} )$  with the same multiplicity.
\end{cor}
\begin{proof} Please note that  \eqref{3.30} follows from \eqref{3.28} and \eqref{3.29} directly. In \eqref{3.30}, let $\lambda R_1$ take place of $R_1$, and we have
\begin{eqnarray}\label{3.31}
 \det(I+\lambda\mathcal{F}(B_0, D))=\det(I+\lambda R_1\cA^{-1}). \nonumber  \end{eqnarray}
This shows that $\sigma(\cF(B_0,D))=\sigma(R_1\cA^{-1} )$. Moreover, by \cite[Theorem 3.5]{HW},  for  $\mathfrak{K}\in\cal{J}(P_N)$, $\lambda_{0}$ is an eigenvalue of $\mathfrak{K}$ of algebraic multiplicity $k$ if and only if $-\lambda_{0}^{-1}$ is zero point of $\det(I+z\mathfrak{K})$ of order $k$. Since $\det(I+\lambda\mathcal{F}(B_0,D))=\det(I+\lambda R_1\cA^{-1})$, the algebraic multiplicity of $\lambda_j\in\sigma(\cF(B_0,D))$ is same as that of $\lambda_j\in\sigma(R_1\cA^{-1} )$.
\end{proof}

In general (\ref{3.3.1}-\ref{3.3.2}) is not right.  For example, let $n=1$, $T=1$,   $B=0_{2}$,
$D=\left(\begin{array}{cc}0 & 1 \\1  & 0 \end{array}\right) $, and let $V_0,V_1$ with frame $(1,0)^T, (0,1)^T$ separately. Easy computation shows that $JD=\left(\begin{array}{cc}-1 & 0 \\0  & 1 \end{array}\right)  $, and $\ga_\lambda(t)=\left(\begin{array}{cc} e^{-\lambda t} & 0 \\0  & e^{\lambda t} \end{array}\right) $, which implies that the corresponding Hamiltonian system has no eigenvalue. Therefore $\prod_{j}(1-\lambda_j^{-1})$ should be understood as $1$ and $\sum_{j} \lambda_j^{-1} =0$. On the other hand, from  the Hill-type formulas (\ref{hillf}), we have
\bea  \det(I-\cF(0_2,D))=e^{-1}\not=\prod_j(1-\lambda_j^{-1}), \quad Tr \cF(0_2,D)=1\not=\sum_j \lambda_j^{-1}. \nonumber\eea

\section{Brake symmetry decomposition for $S$-periodic orbits}
In this section, we will deal with the brake symmetry decomposition. The relationship between the conditional Fredholm determinant of $S$-periodic solutions with brake symmetry and that of the solution of corresponding Lagrangian boundary problem will be considered. At first, we need the following lemma.

\begin{lem} Recall that $\bar E^{\pm}_S$ and $E_S^\pm$ are defined in (\ref{barespm}) and (\ref{espm})
\bea\det(I-\cF(B,D;\bar{E}^\pm_S))=\det(I-\cF(B,D;E^\pm_S)). \nonumber \eea\end{lem}
\begin{proof} We only prove the equality for $E^+_S$. Let $U^{+}: \bar{E}^{+}_S\to E^{+}_S$ be the isomorphic maps, then a function $f$ is an eigenvector of $\cF(B,D;\bar{E}^+_S)$ corresponding to the eigenvalue $\lambda$ if and only if  $U^+f$ is an eigenvector of $\cF(B,D;{E}^+_S)$ corresponding to the same eigenvalue. By \cite{Si},
\begin{eqnarray} {\det} _2(I-\cF(B,D;\bar{E}^+_S))={\det} _2(I-\cF(B,D;E^+_S)). \nonumber \end{eqnarray}
Moreover, let $\bar P^+_{N}$ be the orthogonal projection on $\bar{E}_N^+=\bar E^+_S\cap \oplus_{|\nu|\leq N}(\ker (A_{\bar E^+_S}-\nu))$, then the conditional trace
$$Tr(\cF(B,D;\bar{E}^+_S))=\lim\limits_{N\to\infty} Tr \( \cF(B,D;\bar{E}^\pm_S)|_{\bar E^+_N}\).$$
Similarly, let $$E_N^+=E^+_S\cap \oplus_{|\nu|\leq N}(\ker (A_{E^+_S}-\nu)),$$ and $P_N$ the orthogonal projection onto $E^+_N$,
Since $U^+A_{\bar E^+_S}=A_{E^+_S}U^+$, $U^+ \bar E^+_N=E^+_N$,
we have
$$
Tr \( \cF(B,D;\bar{E}^\pm_S)|_{\bar E^+_N}\)=Tr \(  \cF(B,D;{E}^\pm_S)|_{ E^+_N}\).
$$
It follows that the conditional trace
\begin{eqnarray}
Tr(\cF(B,D;\bar{E}^+_S))=Tr(\cF(B,D;{E}^+_S)).\nonumber
\end{eqnarray}
By (\ref{eq1.5a}), we have the desired equality.
\end{proof}

By \cite{HS1}, the condition (\ref{4.4}) implies that
 \bea  A|_{E_S}-B-\lambda D=(A|_{E_S}-B-\lambda D)|_{E^+_S}\oplus (A|_{E_S}-B-\lambda D)|_{E^-_S}. \nonumber \eea
 We have
\begin{thm}\label{decom} Under the above conditions, \bea \det(I-\cF(B,D;E_S))=\det(I-\cF(B,D;E^+_S)\cdot \det(I-\cF(B,D;E^+_S)).  \nonumber\eea
and
\bea Tr \cF^k(B,D;E_S)=Tr\cF^k(B,D;E^+_S)+Tr\cF^k(B,D;E^-_S).\nonumber \eea

 \end{thm}

In \cite{HW}, we built  the following Hill-type formula for $S$-periodic solutions
\bea   \det(I-\cF(B,D;S))= \det(S\ga_1(T)-I_{2n})/\det(S\ga_0(T)-I_{2n}). \lb{hills} \eea
In the remaining part of the section, we will show that the right side of (\ref{hills}) could also be decomposed under $\mathbb Z_2$ action.
From the brake symmetry, we have
  \bea S\ga_\lambda(T)=SN\ga^{-1}_\lambda(T/2)N\ga_\lambda(T/2)\nonumber\eea
By changing basis, we suppose $N=diag(I_n,-I_n)$. Since $S\in\Sp(2n)\cap O(2n)$, then we can assume $S=\left(\begin{array}{cc}C & -D \\D  & C \end{array}\right)$, where $C^TD$ is symmetric, and $C^TC+D^TD=I$. Since $NS^T=SN$, then we have   $C=C^T$, $D=D^T$ and $CD=DC$. So we can choose  basis such that both $C$ and $D$ are diagonizied,
and we may assume $C=\cos(\theta), D=\sin(\theta)$ with $\theta=diag(\theta_1,\cdots,\theta_n)$.
We write $S$ as $R_\theta:= \left(\begin{array}{cc}\cos(\theta) & -\sin(\theta) \\ \sin(\theta)  & \cos(\theta) \end{array}\right)$, it is obvious that $R(\theta/2)^2=S$. Easy computation shows that
\bea  R_{\theta/2}NR_{\theta/2}=N, SN=  \left(\begin{array}{cc}\cos(\theta) & \sin(\theta) \\ \sin(\theta)  & -\cos(\theta) \end{array}\right). \nonumber \eea
We have
\bea  \det(S\ga_\lambda(T)-I_{2n})&=&\det(SN\ga_\lambda^{-1}(T/2)N\ga_\lambda(T/2)-I_{2n}  )\nonumber \\   \nonumber    &=&\det(R_{\theta/2}NR_{\theta/2} R_{\theta/2}^{-1}  \ga_\lambda^{-1}(T/2)N\ga(T/2)R_{\theta/2}-I_{2n}  )
\\   \nonumber    &=& \det(N \hat{M}_\lambda^{-1}N\hat{M}_\lambda-I_{2n}),\nonumber
\eea
where we set $\hat{M}_\lambda=\ga_\lambda(T/2)R_{\theta/2}$.  By assume
$\hat{M}_\lambda=\left(\begin{array}{cc}a_\lambda & b_\lambda \\c_\lambda  & d_\lambda \end{array}\right)$, then
\bea   \det(S\ga_\lambda(T)-I_{2n})&=& (-1)^n \det( N\hat{M}_\lambda-\hat{M_\lambda}N ) \nonumber \\  &=& (-1)^n\det \left(\begin{array}{cc}0_n & 2b_\lambda \\ -2c_\lambda  & 0_n \end{array}\right)=(-1)^n2^{2n}\det(b_\lambda)\det(c_\lambda) . \nonumber    \eea
We have \bea  \det(S\ga_1(T)-I_{2n})/\det(S\ga_0(T)-I_{2n})= \frac{\det(c_1)}{\det(c_0)}       \frac{\det(b_1)}{\det(b_0)}. \nonumber\eea
Let $Z_+=diag(I_n,0_n)$ ($Z_-=diag(0_n,I_n)$) be the frame of $V^+(N)$($V^-(N)$), then   a frame $\hat{Z}_\pm$ of $V^\pm(SN)$ can be given by $R_{\theta/2}Z_\pm$.
We have \bea \det(\ga_\lambda(T/2)\hat{Z}_+, Z_+  )=\det(\hat{M}_\lambda Z_+,Z_+  )=(-1)^n\det(c_\lambda), \nonumber  \eea
and  \bea \det(\ga_1(T/2)\hat{Z}_+, Z_+  )/ \det(\ga_0(T/2)\hat{Z}_+, Z_+  )=\det(c_1)/\det(c_0).\nonumber \eea
Similar, we have
\bea \det(\ga_\lambda(T/2)\hat{Z}_-, Z_-  )=\det(\hat{M}_\lambda Z_-,Z_-  )=\det(b_\lambda), \nonumber  \eea  and  \bea \det(\ga_1(T/2)\hat{Z}_-, Z_-  )/ \det(\ga_0(T/2)\hat{Z}_-, Z_-  )=\det(b_1)/\det(b_0).\nonumber \eea
We get the decomposition formula
\bea    \det(S\ga_1(T)-I_{2n})/\det(S\ga_0(T)-I_{2n})=\frac{\det(\ga_1(T/2)\hat{Z}_+, Z_+ )}{\det(\ga_0(T/2)\hat{Z}_+, Z_+ )}
\cdot  \frac{\det(\ga_1(T/2)\hat{Z}_-, Z_-  )}{\det(\ga_0(T/2)\hat{Z}_-, Z_-  )     }.        \nonumber        \eea

\section{Relation with the relative Morse index and stability criteria  }

In this section, we will give the relation of conditional Fredholm determinant and relative Morse index,  moreover we give some new stability criteria for the symmetry periodic orbits.

A simple way to understand the relative Morse index $I(A-B, A-B-D)$   is from the viewpoint of spectral flow.
For reader's convenience, we first give a brief review of the spectral flow.
The spectral flow was introduced by Atiyah, Patodi and
Singer \cite{APS} in their study of index theory on manifolds with
boundary. Let $\{A(\theta),\theta\in[0,1]\}$ be a continuous path of
self-adjoint Fredholm operators on a Hilbert space $\mathcal{H}$. Roughly
speaking, the spectral flow of path $\{A(\theta),\theta\in[0,1]\}$
counts the net change in the number of negative eigenvalues of
$A(\theta)$ as $\theta$ goes from $0$ to $1$, where the enumeration
follows from the rule that each negative eigenvalue crossing to the
positive axis contributes $+1$ and each positive eigenvalue crossing
to the negative axis contributes $-1$, and for each crossing, the
multiplicity of eigenvalue is counted.

We come back to the Hamiltonian systems, suppose $B(s,t)\in C([0,1]\times[0,T], \mathcal{S}(2n))$. For $s\in[0,1]$,  let $B_s\in C([0,T], \mathcal{S}(2n))$.  For such two operators $A-B_0$ and $A-B_1$, we can define the relative Morse index via spectral flow.
In fact, by \cite{HS}, we have,
\begin{eqnarray*}
\cI(A-B_0,A-B_1)=-Sf(\{A-B(s), s\in[0,1]\}). \label{5a.8}
\end{eqnarray*}
We list some fundamental property of the relative Morse  index, detail could found in \cite{HOW}
\\
(I) For $B_0$, $B_1$, $B_2$, then \begin{eqnarray*}
\cI(A-B_0,A-B_1)+\cI(A-B_1,A-B_2)=\cI(A-B_0,A-B_2). \label{a5.10}  \end{eqnarray*}
\\(II)Let $D=B_1-B_0$, and we can simply let $B(s)=B_0+sD$.  Let $\kappa=\{s_0\in[0,1], \ker(A-B(s_0))\neq0\}$,
\begin{eqnarray*}\cI(A-B_0, A-B_1)\leq \sum_{s_0\in\kappa}  \upsilon(A-B(s_0)).  \label{5a.12}  \end{eqnarray*}
\\ (III)
if $D>0$, then $I(A-B,A-B-D)\geq0$. By careful analysis\cite{HS},   the crossing form
\begin{eqnarray} \cI(A-B, A-B-D)=\sum_{s_0\in\kappa\cap[0,1)} \upsilon(A-B(s_0)).  \label{5a.14}\nonumber  \end{eqnarray}
Similarly
\begin{eqnarray} \cI(A-B, A-B+D)=-\sum_{s_0\in\kappa\cap(0,1]} \upsilon(A-B(s_0)).  \label{5a.16} \nonumber \end{eqnarray}
\\(IV) Suppose $D_1\leq D\leq D_2$, then
\begin{eqnarray}\cI(A-B,A-B-D_1)\leq \cI(A-B,A-B-D)\leq \cI(A-B,A-B-D_2).\label{indexn}\nonumber
\end{eqnarray}

Similar with the $S$-periodic boundary conditions \cite{HW}, we have the next theorem
\begin{thm}\label{hillevenodd} Assume $A-B$ and $A-B-D$ are non-degenerate, then  $\det(I-\cF(B,D))>0 $($<0$) if and only if $\cI(A-B,A-B-D)$ is even (odd).
\end{thm}
\begin{proof}
Let $P_N$ be the orthogonal projections onto $V_N=\oplus_{|\nu|\leq N} \ker (A-\nu)$. Then
$$
\det((A-B)(A+P_{0})^{-1})=\lim\limits_{N\to\infty} \det((A|_{V_{n}}-P_{n}BP_{n})(A|_{V_{n}}+P_{0})^{-1}).
$$
and the sign of $\det((A-B)(A+P_{0})^{-1})$ is same as the sign of $\det((A|_{V_{n}}-P_{n}BP_{n})(A|_{V_{n}}+P_{0})^{-1})$ for $N$ large enough. Now by the same argument as \cite[Theorem 6.2]{HW}, we have $\det((A|_{V_{n}}-P_{n}BP_{n})(A|_{V_{n}}+P_{0})^{-1})$ is positive (negative) if and only if the difference of total multiplicity of the negative eigenvalues of $A|_{V_N}-P_N BP_{N}$ and $A|_{V_N}$ is even(odd) for $N$ is large enough.  By the continuousness of the relative Morse index, this is equivalent to that $\cI(A, A-B)$ is even(odd), hence $sign\det((A-B)(A+P_{0})^{-1})=(-1)^{\cI(A, A-B)}$. By the same reason , for $A-B-D$, we also have $sign\det((A-B-D)(A+P_{0})^{-1})=(-1)^{\cI(A, A-B-D)}$. Since $\det(I-\cF(B,D))det((A-B)(A+P_{0})^{-1})=det((A-B-D)(A+P_{0})^{-1})$, we get
$$
sign\det(I-\cF(B,D))=(-1)^{\cI(A,A-B-D)}\cdot(-1)^{-\cI(A,A-B)}=(-1)^{\cI(A-B, A-B-D)},
$$
it's easy to see that $\det(I-\cF(B,D))$ is positive (negative) if and only if $\cI(A-B,A-B-D)$ is even(odd).
\end{proof}

In \cite{HOW}, we had use   trace of $\cF^k(B,D)$ to  nontrivial estimation of  relative Morse index. Although in \cite{HOW}, we deal with the operators with of $S$-periodic case, it is totally same for the Lagrangian boundary conditions. The following theorem is from
\cite{HOW}

\begin{thm}\label{thm5.2} Suppose $A-B$ is non-degenerate. Suppose that there are $D_1,D_2\in \mathcal{B}(2n)$ such that $D_1< D < D_2$, with $D_1<0$, $D_2>0$, if there exists  $k\in 2\mathbb N$, such that  $Tr  \cF^k(B+\nu J,D_j)<1$
for $j=1,2$, then $A-B-D$ is non-degenerate, and moreover $\cI(A-B,A-B-D)=0$.
\end{thm}

 Connected with the trace formula (\ref{tracef}),   We can give a estimation of relative Morse index by the trace of matrices.   As a corollary of Theorem \ref{thm5.2} ,  We have
\begin{cor}
Suppose $A-B$ is non-degenerate. Suppose that there are $D\geq0$,  $Tr(G_1^2)-2Tr(G_2)<1$,  then $A-B-D$ is non-degenerate, and  $\cI(A-B,A-B-D)=0$. Similar for the case $D<0$. \end{cor}

Suppose $x(t)$ is a $T$-periodic solution with the fundamental solutions $\ga(t)$.  $x$ is called (spectral) stable  if $\sigma(\ga(T))\in \mathbb U$, is called hyperbolic if  $\sigma(\ga(T))\cap \mathbb U=\emptyset$.
To estimate the stability, we use the Maslov-type index $i_\omega(\gamma)$, which is essentially same as  the relative Morse index \cite{Lon4}. Roughly speaking, for a continuous path $\gamma(t)\in\Sp(2n)$, $\omega\in \mathbb{U}$, the Maslov-type index $i_\omega(\gamma)$ is defined by the intersection number of $\gamma$ and $\Sp_\omega^0(2n)=\{M\in\Sp(2n)\,|\, \det(M-\omega I_{2n})=0\}$.  Details could be found in \cite{Lon2},\cite{Lon4}, some brief review could be found in \cite{HS1}.
From \cite[Theorem 2.5 and Lemma 4.5]{HS}, we have the following proposition.
\begin{prop}\label{prop5a.2}
For $S=\pm I$, we have
\begin{eqnarray*} \cI(A|_{E_S}, A|_{E_S}-B)=\left\{\begin{array}{ll} i_1(\gamma)+n, & \quad
           {\mathrm if}\; S=I_{2n},  \\
           \\
 i_{-1}(\gamma) & \quad {\mathrm if}\; S=-I_{2n} .\end{array}\right.  \label{5a.18}  \end{eqnarray*}

\end{prop}

Let $e(M)$ be the total number of eigenvalues of $M$ on $\mathbb U$, a simple but useful stability criteria is following
\begin{eqnarray} e(\ga(T))/2\emph{}\geq
 |i_{-1}(\gamma)-i_1(\gamma)|.
 \label{5a.26} \end{eqnarray}
All the above results, for that the relative Morse index equals to Maslov-type index and for the stability criteria,   could be proved
for any $S$ boundary condition with $S\in\Sp(2n)\cap O(2n)$, and details could be found in \cite{HS}, \cite{HOW}.

Consider the linear system
\begin{eqnarray} \dot{z}(t)= JB_1(t)z(t),  z(0)=z(T),\label{5.b1}
\end{eqnarray}
where $B_1=B+D$.
We assume (\ref{5.b1}) satisfied the brake symmetry condition with respect to $N$ as given in \S 4.
From \cite{HW}, for $S=\pm I$, we have decomposition of  the relative Morse index
\bea \cI(A|_{E_S}-B,A|_{E_S}-B-D)=I(A|_{E^+_S}-B, A|_{E^+_S}-B-D) +I(A|_{E^-_S}-B, A|_{E^-_S}-B-D).\lb{dindex}\nonumber\eea
We have
\begin{lem}\label{thm5.3} Assume $B$ and $D$ satisfy brake symmetry with respect to $N$. Suppose $|i_1(\ga_{0})-i_{-1}(\ga_0)|=n$, and  $\cI(A|_E-B, A|_E-B-D)=0$, for $E=E^\pm_S$,  $S=\pm I$, then $\ga_1(T)$ is stable. \end{lem}
\begin{proof}  Please note that  $\cI(A|_E-B, A|_E-B-D)=0$, for $E=E^\pm_S$,  $S=\pm I$ implies $i_\pm(\ga_1)=i_\pm(\ga_0)$, then the result from \eqref{5a.26}.
\end{proof}
As a corollary, we have
\begin{cor}\label{cor5.6}  Suppose $|i_{-1}(\ga_{0})-i_{1}(\ga_0)|=n$,   $D^-\leq D\leq D^+$ with $D^+\geq0$, $D^-\leq0$,  and $Tr \cF^2(B,D^\pm,E)<1$ for $E=E^\pm_{\pm I}$, then  $\ga_1(T)$ is stable.
 \end{cor}

 From Corollary  \ref{cor5.6} and the trace formula, we can give stability criteria for the brake symmetry orbits. This criteria will applied for elliptic Lagrangian orbits, please see the detail in next section.

\section{Stability of elliptic relative equilibria in planar $3$-body problem}

   In this section, we use the trace formula to study the stability of ERE in planar $3$-body problem.  A brief introduction for the stability of ERE is given in Subsection \ref{sec6.1}. The applications on studying  the stability of  the elliptic  Lagrangian orbits and elliptic Euler orbits are given in Subsection \ref{sec6.2} and Subsection \ref{sec6.3} separately.   

\subsection{Brief introduction to the stability of elliptic relative equilibria}\label{sec6.1}
In 2005, Meyer and Schmidt\cite{MS}  strongly used the structure of the central configuration for
the elliptic  relative equilibria and symplectically decomposed the fundamental
solution of the orbits  into two parts,
one of which corresponding to the Keplerian solution and the other is the essential part
of the dynamics, needed for studying the  stability. For the planner three body problem, the only central configurations is case of  the Lagrangian triple and Euler collinear control configurations, which the corresponding  ERE is called elliptic Lagrangian solutions and Elliptic Euler orbits.  In this case, the essential part can be written in the following form.

Let $e$ is
the eccentricity, $t$ be the truly anomaly and $r_e(t)=(1+e\cos(t))^{-1}$.  In the rotating coordinate system and by using the true anomaly as the variable,
Meyer and Schmidt \cite{MS} gave a very useful form of the essential part
\bea  \mathcal{B}_e(t)=\left( \begin{array}{cccc} I_{2} & -J \\
J & I_{2}-r_e(t) R
\end{array}\right),\quad t\in[0,2\pi], \quad e\in[0,1),\label{msf} \eea where   $R$ can be considered as the regularized Hessian of the central configurations.
Thus the corresponding Sturm-Liouville system  is
  \bea -\ddot{y}-2J_2\dot{y}+r_e(t) R y=0.  \label{st}\nonumber \eea

 Let $\ga_{e}(t)$ be the fundamental solution of \eqref{msf}, that is
$$ \dot{\ga}_e(t)=J \mathcal{B}_e(t)\ga_e(t), \quad \ga_e(0)=I.$$
The ERE
is called spectrally stable (or elliptic) if all the eigenvalues of $\ga_e(2\pi)$ belong to the unit circle $\mathbb U$, and it is called linearly stable if moreover $\ga_e(2\pi)$ is semi-simple. By contrast, the  ERE is called hyperbolic if no eigenvalue of $\ga_e(2\pi)$ locates on $\mathbb U$, and is called  elliptic-hyperbolic if only part of eigenvalues locate on $\mathbb U$.

We assume $R=\alpha I_2+\eta \cN$ for $\alpha,\eta\geq0$ with $\cN=diag(1,-1)$, which
include the case of Lagrangian and Euler orbits.
Obviously,  $NR=RN$.  Denote
 $N=diag(\cN,-\cN)$. Direct computation shows that $$ N\cB_e(T-t)=\cB_e(t)N, \quad e\in[0,1),$$ which means the system admits the brake symmetry.
 We have the decomposition formula \cite{HW}
\bea  i_1(\ga_e)+n&=& \cI(A|_{E^+_1}, A|_{E^+_1}-\cB_e) +\cI(A|_{E^-_1}, A|_{E^-_1}-\cB_e),    \nonumber  \\
i_{-1}(\ga_e)&=& \cI(A|_{E^+_{-1}}, A|_{E^+_{-1}}-\cB_e)+\cI(A|_{E^-_{-1}}, A|_{E^-_{-1}}-\cB_e),\nonumber   \eea
and
\bea
dim\ker(\ga_{e}(2\pi)-1)=dim(V^+(N)\cap\ga_{e}(\pi)V^+
(N))+dim(V^-(N)\cap\ga_{e}(\pi)V^-(N)).\lb{kern1}\nonumber\eea
\bea
dim\ker(\ga_{e}(2\pi)+1)=dim(V^-(N)\cap\ga_{e}(\pi)V^+
(N)+dim(V^+(N)\cap\ga_{e}(\pi)V^-(N)).\lb{kerp1}\nonumber\eea
Let $$D_e:=\cB_e-\cB_0=diag(0_2, e\cos(t) r_e(t) R),$$ and denote \bea D^\pm_e:=(D_e\pm |D_e|)/2, \nonumber\eea
where $|D_e|=(D_e^2)^{\frac{1}{2}}$. Then, we have $D^+_e\geq0$ and
$D^-_e\leq0$.
\begin{prop} For $E=E^\pm_i$, $i=1,2$, if $Tr \cF^2(\cB_0, D_e^\pm; E)<1 $ then $\ga_1$ is $\pm 1$ non-degenerate and
\bea i_1(\ga_1)=i_1(\ga_0), \quad  i_{-1}(\ga_1)=i_{-1}(\ga_0). \nonumber \eea
\end{prop}


\subsection{Stability of elliptic Lagrangian orbits}\label{sec6.2}

We will give a new estimation to the left stability region of the Lagrangian orbits. In this case $\alpha=3/2$,
$\eta=\frac{\sqrt{9-\beta}}{2}$, $\beta\in[0,9]$.  Please note that $R$ only depend on $\beta$ and $R_{\beta}>0$ for $\beta\in(0,9]$. We denote $\ga_{\beta,e}$  be the fundamental solutions
 correspondding to $\cB_{\beta,e}$ .
By (55) and (58) in \cite[Lemma 4.1]{HS1}, we obtain
\bea  i_1(\ga_{\beta,e})=0, \qquad \forall \,(\beta,e)\in [0,9]\times [0,1).  \lb{2.31}\nonumber\eea

 \bea i_{-1}(\ga_{\beta,0})=\left\{\begin{array}{ll}2 & \quad
           {\mathrm if}\; \beta\in [0,3/4),  \\
           \\
 0, & \quad {\mathrm if}\; \beta\in[3/4,9] .\end{array}\right.\lb{l3}\nonumber\eea
 Set
\be D_{\beta,e}(t)=\cB_{\beta,e}(t)-\cB_{\beta,0}(t)=e\cos(t) r_e(t)K_\beta ,\nonumber \ee
where $ K_\beta= diag\left(0,0,\frac{3+\sqrt{9-\beta}}{2},\frac{3-\sqrt{9-\beta}}{2}\right)$,   then $ A-\cB_{\beta,e}=A-\cB_{\beta,0}-D_{\beta,e}$.
Let  $\cos^\pm(t)=(\cos(t)\pm |\cos(t)|)/2$, and  denote \bea K^\pm_\beta=\cos^\pm(t)K_\beta, \nonumber \eea
then $$  D^\pm_{\beta,e}=e  r_e(t)K_\beta^\pm.$$
We  denote \bea f^\pm(\beta)=Tr(\cF^2(\cB_{\beta,0},K_{\beta}^-;E_{-1}^\pm)), \lb{fbb} \nonumber\eea which is a positive function.   The following theorem holds true.
\begin{thm}\label{th2.1}
For $\beta\in[0,3/4)$, $\gamma_{\beta,e}(2\pi)$ is spectrally stable if \bea 0\leq e<\min\{  (1+f^\pm(\beta)^{\frac{1}{2}})^{-1} \}. \nonumber\eea
\end{thm}
\begin{proof}
The inequality $0\leq e<\min\{(1+f^\pm(\beta)^{\frac{1}{2}})^{-1} \}$ implies that $Tr(\cF^2(\cB_{\beta,0}, e(1-e)^{-1}K_{\beta}^-;E_{-1}^\pm))<1$, then
$(A-\cB_{\beta,0}-e(1-e)^{-1}K_{\beta}^-)|_{E^{\pm}_{-1}}$ is non-degenerate, hence we have $I(A|_{E^{\pm}_{-1}},A-\cB_{\beta,0}-e(1-e)^{-1}K_{\beta}^-|_{E^{\pm}_{-1}})=I(A|_{E^{\pm}_{-1}},A-\cB_{\beta,0}|_{E^{\pm}_{-1}})$. Moreover we have
\bea i_{-1}(\gamma_{\beta,e}) &=& I(A|_{E^{+}_{-1}},A-\cB_{\beta,e}|_{E^{+}_{-1}})+I(A|_{E^{-}_{-1}},A-\cB_{\beta,e}|_{E^{-}_{-1}})  \nonumber \\
&\geq &I(A|_{E^{+}_{-1}},A-\cB_{\beta,0}-e(1-e)^{-1}K_{\beta}^{-}|_{E^{+}_{-1}})+
I(A|_{E^{-}_{-1}},A-\cB_{\beta,0}-e(1-e)^{-1}K_{\beta}^{-}|_{E^{-}_{-1}}) \nonumber \\
&=& I(A|_{E^{+}_{-1}},A-\cB_{\beta,0}|_{E^{+}_{-1}})+I(A|_{E^{-}_{-1}},A-\cB_{\beta,0}|_{E^{-}_{-1}}) \nonumber \\
&=& i_{-1}(\gamma_{\beta,0})=2.\nonumber \eea So $e(\gamma_{\beta,e})\geq 2|i_{1}(\gamma_{\beta,e})-i_{-1}(\gamma_{\beta,e})|=4.$ This complete the proof.
\end{proof}
To compute $f^\pm(\beta)$,
let $\{e_j\}_{j=1}^4$ be the standard basis of $\R^4$, then the frames of $V^+(N)$ and $V^-(N)$ could be given by $(e_1,e_4)$ and  $(e_2,e_3)$ separately. Obviously, $\ga_{\beta,0}(t)=\exp(J\cB_{\beta,0}t)$.
We first consider $f^+(\beta)$, in this case, the boundary conditions is given by $x(0)\in V^-(N)$ and $x(\pi)\in V^+(N)$.  Then we
have $$P^+=(e_2,e_3, \exp(-J\cB_{\beta,0}\pi)e_1, \exp(-J\cB_{\beta,0}\pi)e_4 ),$$ and $Q^+_d=(e_2,e_3,0,0)$.
Setting  $$\hat{K}_\beta(t)=\exp(-J\cB_\beta(t))JK_\beta\exp(J\cB_\beta(t)),$$
we have
\bea M_1=\int_0^\pi\cos^-(t)\hat{K}_\beta(t)dt=\int_{\pi/2}^{\pi}\cos(t)\hat{K}_\beta(t)dt,  \nonumber\eea
and
\bea M_2 &=& \int_0^\pi\cos^-(t_1)\hat{K}_\beta(t_1)dt_1\int_0^{t_1}\cos^-(t_2)\hat{K}_\beta(t_2)dt_2 \nonumber \\ &=& \int_{\pi/2}^{\pi}\cos(t_1)\hat{K}_\beta(t_1)dt_1\int_0^{t_1}\cos(t_2)\hat{K}_\beta(t_2)dt_2. \nonumber \eea  We have
\bea G^+_1=(P^+)^{-1}M_1Q^+_d, \quad   G^+_2=(P^+)^{-1}M_2Q^+_d,\quad f^+(\beta)=Tr (G^+_1)^2-2 Tr (G^+_2) \nonumber\eea
Similarly,  for
 $f^-(\beta)$, the boundary conditions is given by $x(0)\in V^+(N)$ and $x(\pi)\in V^-(N)$, and $$P^-=(e_1,e_4, \exp(-J\cB_{\beta,0}\pi)e_2, \exp(-J\cB_{\beta,0}\pi)e_3 ),$$ and $Q^-_d=(e_1,e_4,0,0)$. We have
\bea G^-_1=(P^-)^{-1}M_1Q^-_d, \quad   G^-_2=(P^-)^{-1}M_2Q^-_d,\quad f^-(\beta)=Tr (G^-_1)^2-2 Tr (G^-_2). \nonumber\eea

As some basic computation  given in \cite{HOW}, with the help of matlab, we have

\begin{figure}[H]
 \centering
   \includegraphics[height=0.49\textwidth,width=0.65\textwidth]{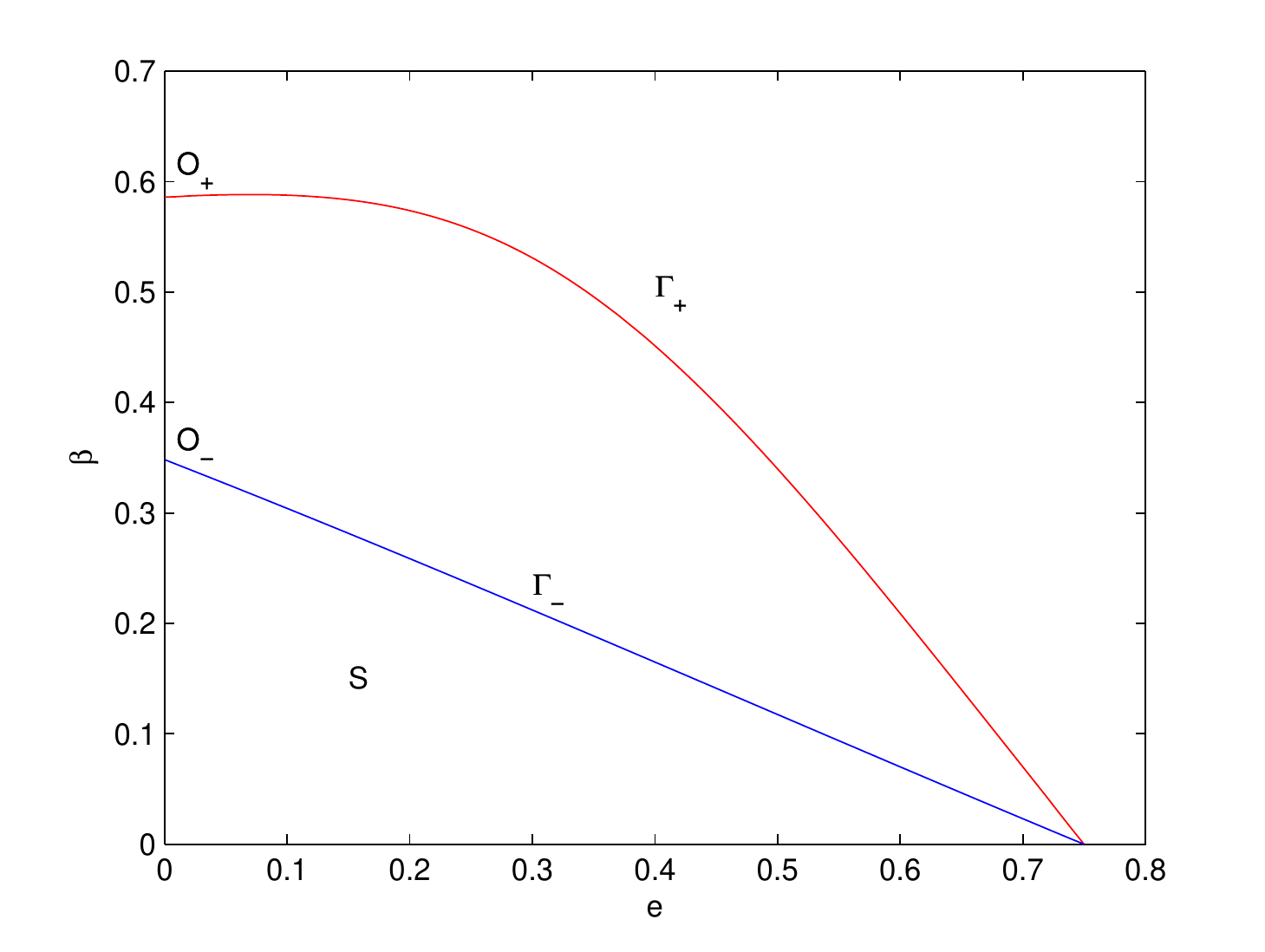}
     \caption{The stable region S.}
\end{figure}
\noindent In Figure 1, the points $O_{-}\approx(0, 0.3483)$, $O_{+}\approx(0, 0.5858)$,
$\Gamma_{-}=\{(\beta,e)|e=1/(1+\sqrt{f^{-}(\beta)})\}$, $\Gamma_{+}=\{(\beta,e)|e=1/(1+\sqrt{f^{+}(\beta)})\}$.

\begin{rem} \label{a5.5}  compare with the result in \cite{HOW}.

\begin{figure}[H]
 \centering
   \includegraphics[height=0.49\textwidth,width=0.65\textwidth]{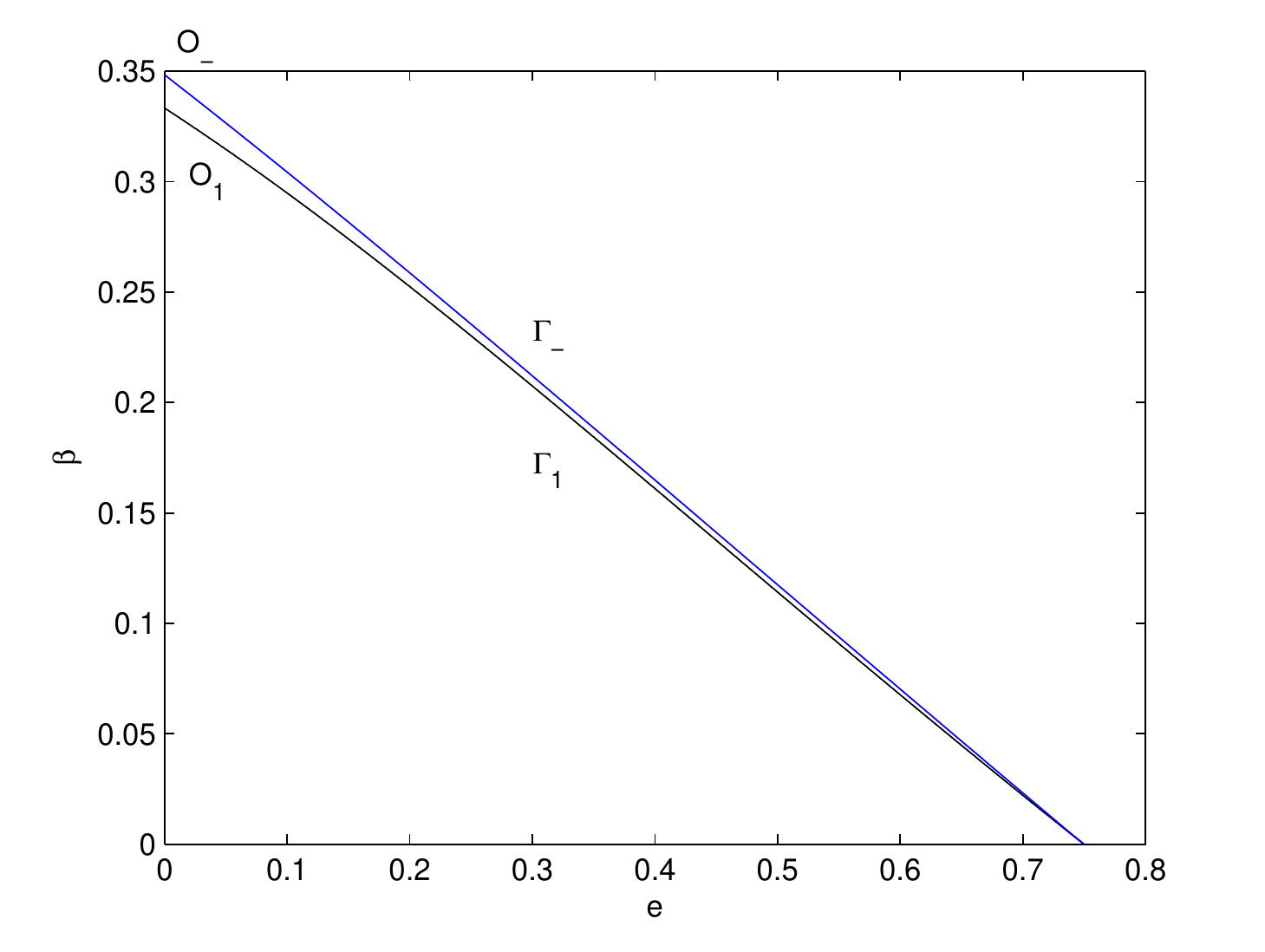}
     \caption{The estimation curve by different methods.}
\end{figure}
\noindent In Figure 2, the points $O_{-}\approx(0, 0.3483)$, $O_{1}\approx(0, 0.3333)$, $\Gamma_{1}$ is given in \cite{HOW}.
From this picture, it easy to know that we can get a better estimation of the stability region by using the trace formulas in this
paper. The reason is that
\bea Tr(\cF^2(\cB_{\beta,0},e(1-e)^{-1}K_{\beta}^-;E_{-1}))
=Tr(\cF^2(\cB_{\beta,0},e(1-e)^{-1}K_{\beta}^-;E^+_{-1})) + Tr(\cF^2(\cB_{\beta,0},e(1-e)^{-1}K_{\beta}^-;E^-_{-1})).\nonumber
\eea
In \cite{HOW}, we need to estimate  $Tr(\cF^2(\cB_{\beta,0},e(1-e)^{-1}K_{\beta}^-;E_{-1}))<1$, but in this paper
we only need to estimate $Tr(\cF^2(\cB_{\beta,0},e(1-e)^{-1}K_{\beta}^-;E^+_{-1}))<1$ and $ Tr(\cF^2(\cB_{\beta,0},e(1-e)^{-1}K_{\beta}^-;E^-_{-1}))<1$. Obviously this condition is weaker, hence we can get a better result.

\end{rem}

\begin{rem} \label{a5.5}
\label{a5.6} In \cite{HLS}, there exist two $-1$-degenerate curves on the stability bifurcation diagram of Lagrangian solution.
This two curves corresponding to spaces $E_{-1}^{+}$ and $E_{-1}^{-}$, hence curves $\Gamma_{+}$ and $\Gamma_{+}$ give a lower bound
of this two $-1$ degenerate curves respectively.
\end{rem}

\subsection{Elliptic-Hyperbolic region of elliptic Euler orbits}\label{sec6.3}

The Euler orbits have been studied in \cite{MSS1}, \cite{LZhou}, \cite{HO}, in this
case,  $R=diag(-\delta,2\delta+3)$, where $\delta\in[0,7]$ only
depends on mass $m_1,m_2,m_3$. Please refer to Appendix A of \cite{MSS1}
for the details.  Although there is no physical meaning for $\delta>7$, we will
assume $\delta\geq 0$ to make the mathematical theory complete.

Let $\ga_{\delta,e}$ be the fundamental solutions of
$\mathcal{B}_{\delta,e}(t)$ which is given by (\ref{msf}), that is
$\dot{\ga}_{\delta,e}=J_2\mathcal{B}_{\delta,e}(t)\ga_{\delta,e}$, $t\in[0,2\pi]$, $\ga_{\delta,e}(0)=I_4$.
 The  stability problem can be studied via the
 Maslov-type index \cite{LZhou},  then we first
review  their results. For any $j\in\mathbb{N}$, there exists
$1$-degenerate curves $\Gamma_j=Gr(\varphi_j(e))$, and we also let
$\Gamma_0=Gr(\varphi_0(e))$ with $\varphi_0(e)=0$.   Then $\ga_{\delta,e}$
only degenerates at $ \cup_{j=1}^\infty \Gamma_j $ and
$dim\ker(\ga_{\delta,e}(2\pi)-I_4)=2 $ if
$(\delta,e)\in\cup_{j=1}^\infty \Gamma_j$. The Maslov-type index
satisfies \bea i_1(\ga_{\delta,e})=2j+3, \quad if\quad
\varphi_j(e)<\delta\leq\varphi_{j+1}(e),\quad j\in \mathbb{N}\cup\{0\}.\nonumber
\eea  Similarly, for $\forall j\in\mathbb{N}$, there exists pair
$-1$-degenerate curves $\Upsilon_j^\pm=Gr(\psi_j^\pm(e))$.

 Let $\psi_j^s(e)=min\{\psi_j^+(e),\psi_j^-(e)\}$ and
$\psi_j^l(e)=max\{\psi_j^+(e),\psi_j^-(e)\}$. Moreover, we set
$\psi_0^l=\psi_0^s=0$,  then for $k\in\mathbb{N}$ we have \bea
i_{-1}(\ga_{\delta,e})=\left\{\begin{array}{ll} 2j, & \quad
           {\mathrm if}\; \delta\in (\psi_{j-1}^l,\psi_{j}^s],  \\
           \\
 2j+1, & \quad {\mathrm if}\; \delta\in(\psi_j^s,\psi_{j}^l] .\end{array}\right.\lb{l3}\nonumber\eea
Direct computation shows that $\psi_j^+(0)=\psi_j^-(0)$, but it is not
clear if, for $e>0$,  there exist other intersection points.  There
is a monotonicity property for Maslov-type index, that is for
$\omega\in\mathbb{ U}$ \bea i_\omega(\ga_{\delta_1,e})\leq
i_\omega(\ga_{\delta_2,e}),\,\ if \,\ \delta_1\leq \delta_2.
\lb{mono}\nonumber\eea
For any $e\in[0,1)$, the $\pm1$ degenerate curves satisfies \bea
0<\psi_1^s(e)\leq\psi_1^l(e)<\varphi_1(e)<\psi_2^s(e)\leq\psi_2^l(e)<\cdots
\psi_j^s(e)\leq\psi_j^l(e)<\varphi_j(e)<\psi_j^s(e)\leq\psi_j^l(e)<\cdots , \lb{e4}\nonumber
\eea
and for $j\in \mathbb{N}$
\bea
\varphi_j(0)=\frac{j-3+\sqrt{9j^4-14j^2+9}}{4},\ \ \psi_j^s(0)=\psi_j^l(0)=\frac{(j+\frac{1}{2})^2-3+\sqrt{9(j+\frac{1}{2})^4-14(j+\frac{1}{2})^2+9}}{4}.\nonumber
\eea
Moreover for the region between the $\pm1$-degenerate curves, $\ga_{\delta,e}(2\pi)$ is elliptic-hyperbolic
and for  the region between the pairs of  $-1$-degenerate curves $\ga_{\delta,e}(2\pi)$
is hyperbolic.

We always set $\psi_k^+$  to be the degenerate curve in the sense that
$V^-(N)\cap\ga_{\delta,e}(2\pi)V^+(N)$ nontrivial and
similarly $\psi_k^-$ to be the degenerate curve in the sense that
$V^+(N)\cap\ga_{\delta,e}(2\pi)V^-(N)$ nontrivial.

Set
\bea D_{\delta,e}(t)=\cB_{\delta,e}(t)-\cB_{\delta,0}(t)=e\cos(t) r_e(t) K_\delta ,\nonumber \eea
where $ K_\delta= diag (0,0,-\delta, 2\delta+3)$. For $t\in[0,\pi]$, we set   $ K^+_\delta(t)= diag (0,0, -\cos^-(t)\delta,\cos^+(t)(2\delta+3))$ and  $ K^-_\delta(t)= diag (0,0,-\cos^+(t)\delta, \cos^-(t)(2\delta+3))$.
then $$  D^\pm_{\delta,e}=e r_e(t) K_\delta^\pm.$$
For $E=E^\pm_{-1}$, let
 $$ g_{1}^\pm(\delta)=Tr \cF^2(\cB_{\delta,0},K_\delta^+,E_{-1}^{\pm} ), \quad  g_{2}^\pm(\delta)=Tr \cF^2(\cB_{\delta,0},K_\delta^-,E_{-1}^{\pm} ). $$

\begin{thm}\label{th6.5}
For $\delta\in[0,\psi_1^s(0))$, $\gamma_{\delta,e}(2\pi)$ is elliptic-hyperbolic, if \bea 0\leq e<\min\{ (1+g_1^\pm(\delta)^{\frac{1}{2}})^{-1}\}\nonumber \eea
\end{thm}
\begin{proof}
The inequality $0\leq e<\min\{(1+g_1^\pm(\delta)^{\frac{1}{2}})^{-1}\}$ implies that $Tr(\cF^2(\cB_{\delta,0},\frac{e}{1-e}K_{\delta}^+;E_{-1}^\pm))<1$, then
$(A-\cB_{\delta,0}-\frac{e}{1-e}K_{\delta}^+)|_{E^{\pm}_{-1}}$ is non-degenerate, hence we have $ I(A|_{E^{\pm}_{-1}},A-\cB_{\delta,0}-\frac{e}{1-e}K_{\delta}^+|_{E^{\pm}_{-1}})=I(A|_{E^{\pm}_{-1}},A-\cB_{\delta,0}|_{E^{\pm}_{-1}})$. Then
\bea i_{-1}(\gamma_{\delta,e}) &=& I(A|_{E^{+}_{-1}},A-\cB_{\delta,e}|_{E^{+}_{-1}})+I(A|_{E^{-}_{-1}},A-\cB_{\delta,e}|_{E^{-}_{-1}}) \nonumber \\
& \leq & I(A|_{E^{+}_{-1}},A-\cB_{\delta,0}-\frac{e}{1-e}K_{\delta}^{+}|_{E^{+}_{-1}})+
I(A|_{E^{-}_{-1}},A-\cB_{\delta,0}-\frac{e}{1-e}K_{\delta}^{+}|_{E^{-}_{-1}}) \nonumber \\
&=& I(A|_{E^{+}_{-1}},A-\cB_{\delta,0}|_{E^{+}_{-1}})+I(A|_{E^{-}_{-1}},A-\cB_{\delta,0}|_{E^{-}_{-1}}) \nonumber \\
&=& i_{-1}(\gamma_{\delta,0})=2. \nonumber\eea So $i_{-1}(\gamma_{\delta,e})$ does not increase and it implies that the region
$\{(\delta,e)|\ \ 0\leq e<\min\{(1+g_1^\pm(\delta)^{\frac{1}{2}})^{-1}\}\}$ is between curves $Gr(\psi_{0}^l(e))$ and $Gr(\psi_{1}^{s}(e))$, from \cite{LZhou},
we know it's elliptic-hyperbolic between this two curve.
\end{proof}
\begin{thm}\label{th6.6}
For $\delta\in(\psi_j^s(0),\psi_{j+1}^s(0)), j\in \mathbb{N}$, $\gamma_{\delta,e}(2\pi)$ is elliptic-hyperbolic, if
\bea 0\leq e<\min\{(1+g_1^\pm(\delta)^{\frac{1}{2}})^{-1},(1+g_2^\pm(\delta)^{\frac{1}{2}})^{-1}\}\nonumber
\eea
\end{thm}
\begin{proof}
For $0\leq e<\min\{(1+g_1^\pm(\delta)^{\frac{1}{2}})^{-1}\}$, like the proof of Theorem \ref{th6.5}, we get
 $i_{-1}(\gamma_{\delta,e})$ does not increase. For $0\leq e<\min\{(1+g_2^\pm(\delta)^{\frac{1}{2}})^{-1}\}$,
like the proof of Theorem \ref{th2.1}, we get index $i_{-1}(\gamma_{\delta,e})$ does not decreasing. So
the region $\{(\delta,e)|\ \ 0\leq e<\min\{1/(1+\sqrt{g_{1}^{\pm}(\delta)}),1/(1+\sqrt{g_{2}^{\pm}(\delta)})\}\}$ must between
curves $Gr(\psi_{j}^l(e))$ and $Gr(\psi_{j+1}^{s}(e))$, from \cite{LZhou},
we know it's elliptic-hyperbolic between this two curve.
\end{proof}
For Lagrangian solution, we have given the deals in computing the function $f^{\pm}(\beta)$. By the same way, we also can compute
$g_{1}^{\pm}(\delta)$ and $g_{2}^{\pm}(\delta)$. With the help of matlab, we have the estimation of the elliptic-hyperbolic(\textbf{EH}) region
of Euler solution.
\begin{figure}[H]
 \centering
   \includegraphics[height=0.49\textwidth,width=0.65\textwidth]{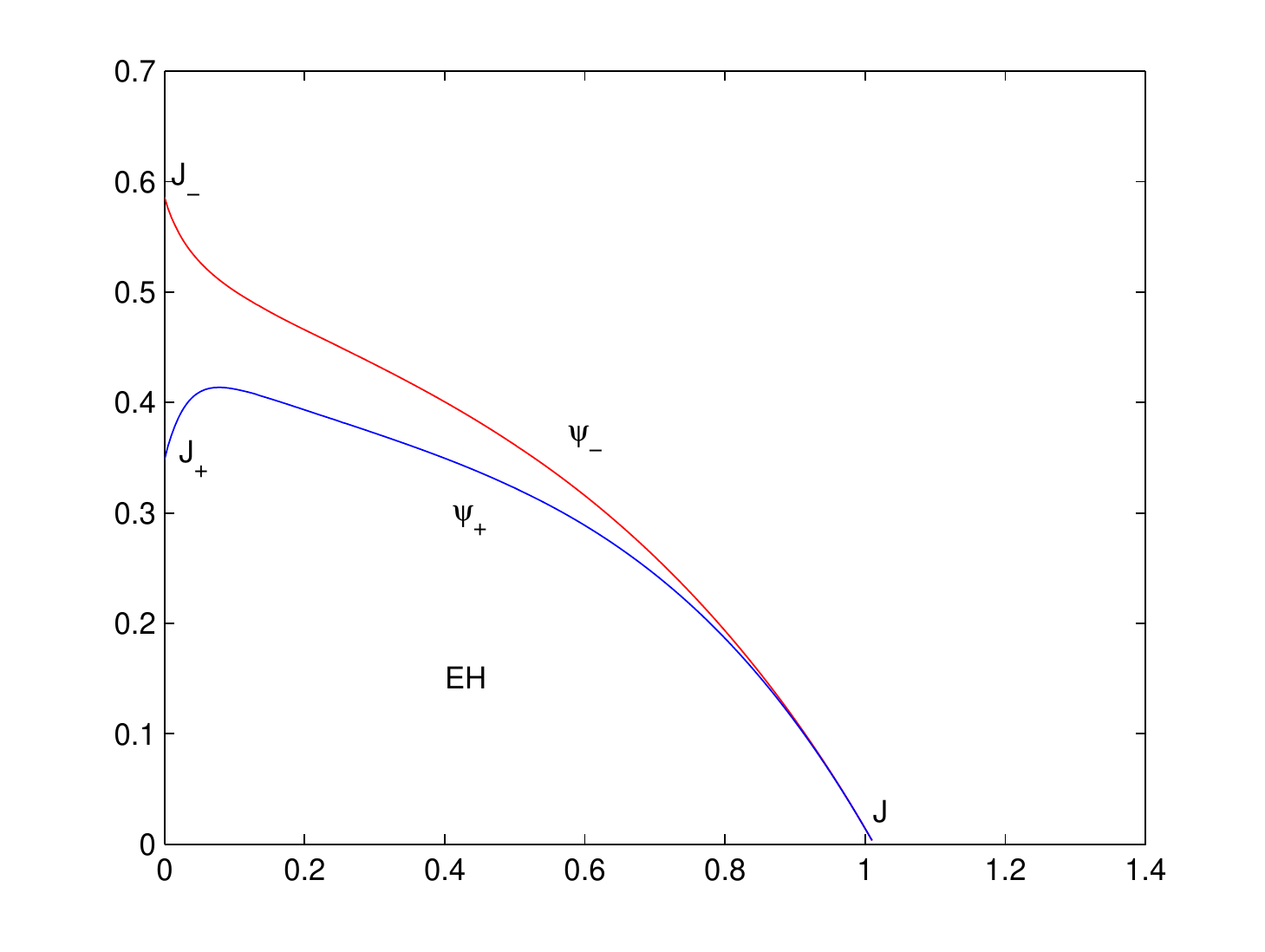}
     \caption{The EH region for the Euler solution with $\delta\in[0,\psi_{1}^s(0))$.}
\end{figure}
In Figure 3, the points $J_{+}\approx(0, 0.3483)$, $J_{-}\approx(0, 0.5858)$,
$\psi_{-}=\{(\beta,e)|e=1/(1+\sqrt{g_{1}^{-}(\beta)})\}$, $\psi_{+}=\{(\beta,e)|e=1/(1+\sqrt{g_{1}^{+}(\beta)})\}$.
\\

\noindent {\bf Acknowledgements.} The  authors sincerely thank Y. Long for his encouragements and interests.

\end{document}